\documentclass[a4paper,9pt,twoside]{book}
\usepackage[english]{babel}
\usepackage{amsthm}
\usepackage[utf8]{inputenc}
\usepackage{amssymb}
\usepackage{amsmath}
\usepackage{cite}
\usepackage{color}
\usepackage{hyperref}
\usepackage{titlesec}
\usepackage{fancyhdr}
\usepackage{bbm}
\usepackage{tikz}
\usepackage{tkz-berge}
\usepackage{float}
\usepackage{subfig}
\usepackage{listings}
\fancypagestyle{plain}{ %
  \fancyhf{} 
  
}

\renewenvironment{proof}{{\scshape \setlength{\parindent}{1.5em}\indent Proof.}}{\qed \vspace{0.5em} \par }
\titleformat{\chapter}[block]{\Large\bfseries\centering}{\thechapter}{0.5em}{}
\titleformat{\subsection}[hang]{\bfseries\centering}{\thesubsection}{0.5em}{}
\titleformat{\section}[hang]{\bfseries\centering}{\thesection}{0.5em}{}
\hypersetup{colorlinks=false}
\newtheoremstyle{plain_own} 
    {\topsep}                    
    {\topsep}                    
    {\itshape}                   
    {1.5em}                           
    {\scshape}                   
    {.}                          
    {.5em}                       
    {}  
\newtheoremstyle{definition_own} 
    {\topsep}                    
    {\topsep}                    
    {}             		      
    {1.5em}                           
    {\scshape}                   
    {.}                          
    {.5em}                       
    {}  
\theoremstyle{plain_own}
\newtheorem{lemma}{Lemma}[section]
\newtheorem{theorem}[lemma]{Theorem}
\newtheorem{corollary}[lemma]{Corollary}
\theoremstyle{definition_own}
\newtheorem{definition}[lemma]{Definition}
\newtheorem{remark}[lemma]{Remark}
\newtheorem{example}[lemma]{Example}

\begin{document}
\begin{titlepage}
\begin{center}
\textbf{\Huge{BACHELOR THESIS}}
\vskip 2cm
\textbf{\huge Isomorphism Classes\vskip 0.4cm
of\vskip 0.4cm Vertex-Transitive \vskip 0.9cm Tournaments}
\vskip 1cm
\begin{align*}
\textnormal{\Large Author}&\textnormal{\Large: Stefan Zetzsche}\\
\textnormal{\Large ID}&\textnormal{\Large:}\\
	\textnormal{\Large Supervisor}&\textnormal{\Large: Dr. Matthias Hamann}\\
	\textnormal{\Large Faculty}&\textnormal{\Large: Natural Science, Computer Science and Mathematics}\\
	\textnormal{\Large Course}&\textnormal{\Large: B.Sc. Mathematics}\\
	\textnormal{\Large Submitted}&\textnormal{\Large: 2016}
\end{align*}
\vskip 6cm
\textnormal{\Large University of Hamburg, Bundesstr. 55, 20146 Hamburg}\\
\end{center}
\end{titlepage}

\chapter*{Abstract}
Tournaments are graphs obtained by assigning a direction for every edge in an undirected complete graph. We give a formula for the number of isomorphism classes of vertex-transitive tournaments with prime order. For that, we introduce Cayley tournaments, which are special Cayley digraphs, and show that there is a one to one identification between the isomorphism classes of vertex-transitive tournaments with prime order and the isomorphism classes of Cayley tournaments with prime order.
\tableofcontents

\chapter{Introduction}
In his paper \cite{turner1967point} from 1967, Turner enumerated the isomorphism classes of vertex-transitive graphs of prime order. In fact, every vertex-transitive graph of prime order is isomorphic to a Cayley graph. Turner showed that two Cayley graphs $\textnormal{Cay}(\mathbb{Z}_p,S)$ and $\textnormal{Cay}(\mathbb{Z}_p,S')$ are isomorphic if and only if $S'=aS \textnormal{ for some } a \in \mathbb{Z}_p^{\times}$. Hence, the number of isomorphism classes is the number of equivalence classes of the relation '$S \textnormal{ equivalent to } S'$', where $S$ is equivalent to $S'$ if and only if $S'=aS \textnormal{ for some } a \in \mathbb{Z}_p^{\times}$. To enumerate the equivalence classes, he used Pólya's enumeration theorem in the following way. The set of connection sets $S$ can be identified with the set of characteristic functions $X = \lbrace \varphi \colon \mathbb{Z}_p \setminus \lbrace 0 \rbrace \rightarrow \lbrace 0, 1 \rbrace \rbrace$. We know that every $a \in \mathbb{Z}_p^{\times}$ permutes  $\mathbb{Z}_p \setminus \lbrace 0 \rbrace$ and we can expand this to a permutation of $X$ by defining $a\varphi(d) = \varphi(a^{-1}d)$. Pólya's enumeration theorem then yields the number of orbits of $X$ under $\mathbb{Z}_p \setminus \lbrace 0 \rbrace$, which equals the number of equivalence classes.

Alspach and Mishna \cite{alspachmishna} generalised this approach in 2002. For that, a particular Cayley (di)graph $\textnormal{Cay}(\Gamma,S)$ is called \textit{(D)CI-(di)graph}, if $\textnormal{Cay}(\Gamma,S)$ is isomorphic to $\textnormal{Cay}(\Gamma, S')$ if and only if $S'=aS$ for some group automorphism $a \in \textnormal{Aut}(\Gamma)$. Furthermore, $\Gamma$ is called \textit{(D)CI-group}, if for arbitrary $S,S',\ \textnormal{Cay}(\Gamma,S)$ is isomorphic to $\textnormal{Cay}(\Gamma, S')$ if and only if $S'=aS$ for some group automorphism $a \in \textnormal{Aut}(\Gamma)$. With these definitions, Alspach and Mishna were able to extend Turner's result to Cayley (di)graphs on cyclic groups of square-free order, circulant (di)graphs of arbitrary order and Cayley (di)graphs on elementary abelian groups. Unfortunately, it is not possible to use their approach to enumerate the isomorphism classes of vertex-transitive tournaments of prime order. Though, following Turner, we are able to show that every such graph is isomorphic to a special kind of Cayley digraph, that is a Cayley tournament. Even more, we state, that a Cayley tournament of $\mathbb{Z}_p$ is a DCI-digraph. Therefore, we also also seek to enumerate the related equivalence classes, still though, in a slightly different spirit. In chapters two and three we introduce all the necessary concepts, that is group actings, vertex-transitivity and Cayley graphs. In particular, we show that any undirected, connected and vertex-transitive graph is a retract of a Cayley graph. Furthermore, we seek to understand, in which case a graph is isomorphic to a Cayley graph. The main result is then given by Theorem 4.1.3. Finally, the last chapter investigates a possible generalisation of our approach to groups of arbitrary order. Precisely, we look into Mansilla's \cite{mansilla2004infinite} construction of infinite families of vertex-transitive tournaments, that are not Cayley tournaments and a paper by Joy Morris \cite{morris1999isomorphic}, investigating in which case two Cayley graphs on non-isomorphic groups are isomorphic. 

\chapter{Algebraic Tools}
To begin with, we will clear the notations. Graphs will be denoted by $G,H$ and its vertex set and edge set by $V(G)$ and $E(G)$, respectively. Vertices are denoted by $v,w$ and an edge from $v$ to $w$ by $v \sim w$. 
To prevent misunderstandings, we will simply write out that two elements are equivalent, instead of also using $\sim$. For groups we write $\Gamma$, for subsets most of the time $S$ and for its elements $\rho, \sigma, g, h, a, b$. Though, since Cayley graphs' vertex sets are also groups, we will sometimes have to switch between the notations. The same goes for sets, which we name by $V,W$ or $X,Y$, where $V$ indicates the further use as a vertex set. The order of a set $X$ goes by $|X|$, if we write $|G|$, we mean $|V(G)|$. Furthermore, $\Gamma' < \Gamma$ indicates that $\Gamma'$ is a subgroup of $\Gamma$. Instead of $\mathbb{Z}/n\mathbb{Z}$, we will write $\mathbb{Z}_n$. The multiplicative group of units is denoted by $\mathbb{Z}_n^{\times}$.  Two graphs $G,H$ are homomorphic, if there exists a map $\varphi \colon V(G) \rightarrow V(H)$, such that, if $v \sim w$, also $\varphi(v) \sim \varphi(w)$. The map $\varphi$ is then called (graph-)homomorphism. If there further exists a homomorphism $\varphi^{-1} \colon V(H) \rightarrow V(G)$, such that $\varphi \circ \varphi^{-1} = \textnormal{id}_{V(H)}$ and $\varphi^{-1} \circ \varphi = \textnormal{id}_{V(G)}$, $G$ and $H$ are called isomorphic and we will write $G \cong H$. In this case $\varphi$ and $\varphi^{-1}$ are called (graph-)isomorphism. Furthermore, if $G=H$, we will call them (graph-)automorphism. The automorphism group is given by $\textnormal{Aut}(G)$. The dot on $\stackrel{.} \bigcup$ indicates a disjunct union. Most of the time we will use multiplicative notation for groups, even though the main application is the additive group $\mathbb{Z}_n$. All the groups we use are finite, if not stated otherwise. $\mathbb{F}_n$ is the unique field of order $n$. $\textnormal{deg}^+(v)$ and $\textnormal{deg}^-(v)$ denote the outer and the inner degree of some vertex $v$, respectively. The projection on the $i$-th coordinate is given by $\textnormal{pr}_i$.
\section{Permutation groups}
In this section we will basically follow \cite{godsil2013algebraic}. Let $V$ be a set. The group of all permutations on $V$, that is the group of all bijections on $V$, is denoted by $\textnormal{Sym}(V)$, or simply $\textnormal{Sym}(n)$, if $\vert V \vert = n \in \mathbb{N}$. Any subgroup of $\textnormal{Sym}(V)$ is called \textit{permutation group} on $V$. Hence, if $V$ is the vertex set of a graph $G$, the group of automorphisms of $G$ is a permutation group on $V$.

First we will recollect some basic constructions from algebra. Let $\Gamma$ be a group and $V$ a set. If there exists a map $\varphi \colon \Gamma \times V \rightarrow V$, such that 
\[ \varphi(1,v) = v \; \textnormal{for all}\ v \in V \]
and
\[ \varphi(g, \varphi(h,v)) = \varphi(gh,v) \, \textnormal{for all}\ g,h \in \Gamma, v \in V, \] 
we say $\Gamma$ \textit{acts} on $V$. $\varphi$ induces a group homomorphism, a so called \textit{permutation representation} $ \rho:\Gamma \rightarrow \textnormal{Sym}(V)$ by 
\[ \rho(g)(v) := \varphi(g,v)\]
 and vice versa. For $v \in V,\, g \in \Gamma$, we shorten the notation by 
 \[v^g := \rho(g)(v).\] 
 Then for every subset $W \subseteq V$, $W^g$ is a subset and every $g \in \Gamma$ induces a permutation on $\mathcal{P}(V)$. We say $H \subseteq V$ is $\Gamma$-\textit{invariant}, if $H^g = H \, \textnormal{for all}\ g \in \Gamma$. Such subsets allow us to introduce $g \downharpoonright H$, the restriction of the induced permutation of $g$ to $H$. Again, the image of $\Gamma \rightarrow \textnormal{Sym}(H), \, g \mapsto g \downharpoonright H$ is a permutation group on $H$. We say a permutation group $\Gamma$ is \textit{transitive}, if for any $v, w \in V$, we find $g \in \Gamma$, such that $v^g = w$. 

For $v \in V$ consider the \textit{orbit} $v^\Gamma := \lbrace  v^g \mid g \in \Gamma \rbrace \subseteq V$. Clearly, it is $\Gamma$-invariant. Even more, we have
\[ (v^g)^{hg^{-1}} = v^{hg^{-1}g} = v^h ,\] 
hence $\Gamma$ acts transitively on $v^{\Gamma}$. Let $w \in v^{\Gamma}$, then $w=v^{g},\ g\in \Gamma$ and therefore $w^h = (v^{g})^h = v^{hg} \in v^{\Gamma}$, for $h\in \Gamma$. On the other hand, $v^h = v^{hg^{-1}g} = (v^{g})^{hg^{-1}} = w^{hg^{-1}} \in w^{\Gamma}$. Hence, it holds $w^{\Gamma} = v^{\Gamma}$ and for $w \not\in v^{\Gamma}$ it also holds $w^{\Gamma} \cap v^{\Gamma} = \emptyset$. Thereby it follows $V=\, \stackrel{.} \bigcup_{v} v^{\Gamma}$ and hence, that any $\Gamma$-invariant subset is an union of orbits of $\Gamma$. This is the same as acknowledging that
\[ v\ \textnormal{equivalent to } w\ \textnormal{if and only if}\ \exists g \in \Gamma: v^g = w \] defines an equivalence relation on $V$. Hence, $\Gamma$ acts transitively on $V$, if and only if there exists only one orbit.

\section{Counting}
If $\Gamma$ is a permutation group on $V$, for any $v \in V$, we define the \textit{stabilizer} $\Gamma_v := \lbrace g \in \Gamma \mid v^g = v \rbrace.$ One generalizes this notion by $\Gamma_{v_1,...,v_k} := \bigcap_i \Gamma_{v_i}$ and weakens it for any $W \subseteq V$ with the concept of the \textit{setwise stabilizer} $\Gamma_W := \lbrace g \in \Gamma \mid W^g = W \rbrace$. If $G$ is a group and $H\subseteq$ a subgroup, the relation 
\[
a\ \textnormal{equivalent to } b\ \textnormal{if and only if } \exists h \in H: a = bh\] defines an equivalence relation on $G$. The equivalence class of an element $b$ is denoted by $bH$ and called the \textit{left coset} of $H$ in $G$ with respect to $b$. Since the stabilizer $\Gamma_v$ is a subgroup of $\Gamma$, we get a partition of $\Gamma$ into left cosets of $\Gamma_v$.
\begin{lemma}
\label{bijec}
	Let $\Gamma$ be a permutation group acting on $V$ and let $W \subseteq V$ be an orbit of $\Gamma$. Then for any $v,w \in W$ there exists $h \in \Gamma$, such that $\lbrace g \in \Gamma \mid v^g = w \rbrace = h\Gamma_{v}$. Conversely, for any $h \in \Gamma,\, v\in W$, there exists $w \in W$, such that for all $g \in h\Gamma_v: v^g=w$.
\end{lemma}
\begin{proof}
	Let $v,w \in W$. Since $W$ is an orbit of $\Gamma$, $\Gamma$ acts transitively on $W$. Hence, one finds $h \in \Gamma$, such that $v^h = w$. Suppose there is another $g \in \Gamma$, such that $v^g = w$. Then $v^h = v^g$, hence $v^{h^{-1}g} = (v^g)^{h^{-1}} = (v^h)^{h^{-1}} = v^{h^{-1}h} = v$. Therefore $h^{-1}g \in \Gamma_v$ and $g \in h\Gamma_v$. Conversely, let $g \in h \Gamma_v$. Then $v^g = v^h =: w$.
\end{proof}
\begin{lemma}[Orbit-stabilizer]
\label{orbitstabilizer}
	Let $\Gamma$ be a permutation group acting on $V$ and let $v \in V$. Then
	\begin{equation*}
		\vert \Gamma_v \vert \vert v^\Gamma \vert = \vert \Gamma \vert.
	\end{equation*}
\end{lemma}
\begin{proof}
	For a fixed $v \in V$ set $W := v^{\Gamma}$. From the previous \autoref{bijec}, we know that there is a bijection between $W$ and the left cosets of $\Gamma_v$. Since $\Gamma = \stackrel{.} \cup_{h} h \Gamma_v$, we know that $\Gamma$ can be partitioned into $\vline W \vline = \vline v^\Gamma \vline$ left cosets of $\Gamma_v$, each containing $\vline \Gamma_v \vline$ elements. 
\end{proof}
How are the stabilizers of two distinct points $v,w \in V$ related? Following Lemma answers this for the case in which $v,w$ are in the same orbit.

\begin{lemma}
	Let $\Gamma$ be a permutation group acting on $V$ and let $v \in V$. If $g \in \Gamma$, then $g\Gamma_v g^{-1} = \Gamma_{v^g}$.
\end{lemma}
\begin{proof}
	Let $ghg^{-1} \in g\Gamma_v g^{-1}$. Then $(v^g)^{ghg^{-1}} = v^{ghg^{-1}g} = v^{gh} = (v^{h})^{g} = v^g$.
	Conversely, if $h \in \Gamma_{v^g},$ we get $v^g = (v^g)^h = v^{hg}$ and therefore $v = (v^g)^{g^{-1}} = (v^{hg})^{g^{-1}} = v^{g^{-1}hg}$. Hence $g^{-1}hg \in \Gamma_v$ or $h \in g\Gamma_v g^{-1}$.
\end{proof}
We will now introduce a dual concept to the stabilizer. For $g \in \Gamma$, the fixed points by $g$  are $\textnormal{fix}(g):=\lbrace v \in V \mid v^g = v \rbrace$. The following Lemma by Burnside states that the number of orbits of $\Gamma$ on $V$ is equal to the average number of fixed points by any element $g \in \Gamma$.
\begin{lemma}[Burnside]
Let $\Gamma$ be a permutation group acting on $V$. Then 
\begin{equation*}
	\vert \lbrace v^\Gamma \mid v \in V \rbrace \vert \vert \Gamma \vert = \sum_{g \in \Gamma} \vert \textnormal{fix}(g) \vert .
\end{equation*}
\end{lemma}
\begin{proof}
	Counting $g$ gives $\vert \lbrace (g, v) \mid g \in \Gamma, \, v \in V, \, v^g = v \rbrace \vert =  \sum_{g \in \Gamma} \vert \textnormal{fix}(g)\vert$. On the other hand counting $v$ gives $\vert \lbrace (g, v) \mid g \in \Gamma, \, v \in V, \, v^g = v \rbrace \vert =  \sum_{v \in V} \vert \Gamma_v \vert$. We have $\vert \Gamma_v \vert = \vert \Gamma_{v^g} \vert$ for any $g \in \Gamma$, by $h \mapsto ghg^{-1}$. And hence $\sum_{v \in v^\Gamma} \vert \Gamma_v \vert = \vert v^\Gamma \vert \vert \Gamma_v \vert = \vert \Gamma \vert$, where for the last equality we use the Orbit-stabilizer \autoref{orbitstabilizer}. Hence $\sum_{g \in \Gamma} \vert \textnormal{fix}(g)\vert = \sum_{v \in V} \vert \Gamma_v \vert = \vert \lbrace v^\Gamma \mid v \in V \rbrace \vert \vert \Gamma \vert$, which proves the claim.
\end{proof}

\chapter{Cayley graphs}	
\section{Vertex-transitive graphs}
We say a graph $G$ is \textit{vertex-tansitive}, if $\textnormal{Aut}(G)$ acts transitively on $V(G)$, i.e. for any $v,w \in V(G)$, we find $\rho \in \textnormal{Aut}(G)$, such that $v^\rho = w$. We will begin this section with some simple examples.\\\\
	For $k \in \mathbb{N}$ let the $\textit{k-cube}$ $Q_k$ be the graph with vertex set $\mathbb{Z}_2^k$ and $v \sim w$ if and only if for some $i, 1 \leq i \leq k: v_i \not= w_i$, but $v_j = w_j, \,$ for all $j, 1 \leq j \leq k,\, j\not=i$.
\begin{lemma}
	The $k$-cube $Q_k$ is vertex transitive.
\end{lemma}
\begin{proof}
	For $v \in \mathbb{Z}_2^k$ consider the permutation of $V(Q_k)$, $\rho_v: w \mapsto w + v$. Since $w,u \in V(Q_k)$ differ in exactly one position if and only if $w+v$ and $u+v$ differ in exactly one position, $\rho$ is also an automorphism. The set $\Gamma = \lbrace \rho_v \mid v \in G \rbrace$ forms a subgroup of $\textnormal{Aut}(Q_k)$. It acts transitively, because for any $v,w \in V$ the map $\rho_{-v+w}$ sends $v$ to $w$. 
\end{proof}

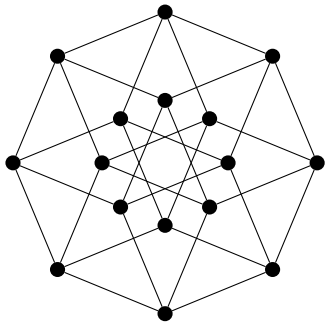
\begin{figure}[H]
\centering
\begin{tikzpicture}
[nodeDecorate/.style={shape=circle,inner sep=2pt,outer sep=0pt,fill},%
  lineDecorate/.style={-,thin},%
  scale=2]
\foreach \nodename/\x/\y in {1/1/0, 2/0.7071/0.7071, 3/0/1,
  4/-0.7071/0.7071, 5/-1/0, 6/-0.7071/-0.7071, 7/0/-1,
  8/0.7071/-0.7071, 9/0.4142/0, 10/0.2928/0.2928, 11/0/0.4142,
  12/-0.2928/0.2928, 13/-0.4142/0, 14/-0.2928/-0.2928, 15/0/-0.4142,
  16/0.2928/-0.2928}
{
  \node (\nodename) at (\x,\y) [nodeDecorate] {};
}
\path
\foreach \startnode/\endnode in {1/2, 1/8, 1/10, 1/16, 2/3, 2/9, 2/11,
  3/4, 3/10, 3/12, 4/5, 4/11, 4/13, 5/6, 5/12, 5/14, 6/7, 6/13, 6/15,
  7/8, 7/14, 7/16, 8/9, 8/15, 9/12, 9/14, 10/13, 10/15, 11/14, 11/16,
  12/15, 13/16}
{
  (\startnode) edge[lineDecorate] node {} (\endnode)
};
\end{tikzpicture}
 \caption{The $4$-cube $Q_4$}
\end{figure}

The \textit{cycle} on \textit{n} vertices is the graph $C_n$ with vertex set $\mathbb{Z}_n$ and $v \sim w$ if and only if $v-w \equiv \pm 1$ mod $n$. If $\rho \in 	\textnormal{Sym}(n)$, such that $\rho(v) \equiv v+1$ mod $n$, it follows that $\rho \in \textnormal{Aut}(C_n)$. Hence $\textnormal{Aut}(C_n)$ must contain the cyclic subgroup $H := \lbrace \rho^k \mid 0 \leq k \leq n-1 \rbrace.$ For any two $v,w \in V(C_n)$ one finds a $\rho^k \in H$, such that $\rho^k(v) = w$, i.e. $H$ acts transitively on $C_n$. 

\begin{definition}
\label{cayleydef}
Let $\Gamma$ be a group and let $S \subset \Gamma$ with $1 \not\in S$. The \textit{Cayley digraph} $\overrightarrow{\textnormal{Cay}}(\Gamma,S)$ is the graph with vertex set $\Gamma$ and edge set $\lbrace (g, gs) \mid g,s \in \Gamma \rbrace$. If, in addition, $S=S^{-1}$, the \textit{Cayley graph} ${\textnormal{Cay}}(\Gamma,S)$ is defined. On the other hand, if $S \cap S^{-1} = \emptyset$ and $S \cup S^{-1} = \Gamma \setminus \lbrace 1 \rbrace$, the \textit{Cayley tournament} $\overrightarrow{\textnormal{Cay}}_{\textnormal{tourn}}(\Gamma,S)$ is defined. In any case, the set $S$ is called the \textit{connection set}.
\end{definition}

It follows $g \sim h$ if and only if $g^{-1}h \in S$. If $1 \in S$, we would get loops since $g \sim g$, for all $g \in \Gamma$.

\begin{definition}
	A \textit{tournament} is a directed graph obtained by assigning a direction for each edge in an undirected complete graph. 
\end{definition}
\begin{remark}
\label{cayleytourn}
The Cayley tournament is a tournament.	
\end{remark}
\begin{proof}
	Let $\Gamma$ be a group. Let $S \subset \Gamma$ with $1 \not\in S$, $S \cap S^{-1} = \emptyset$, $S \cup S^{-1} = \Gamma \setminus \lbrace 1 \rbrace$. Let $g, h \in \Gamma, g \not = h$. The latter yields $g^{-1}h \not = 1$, hence $g^{-1}h \in S$ or $g^{-1}h \in S^{-1}$. Since $S \cap S^{-1} = \emptyset$, only one of both holds. The first one yields $g \sim h$. On the other hand, if $g^{-1}h \in S^{-1}$, we have $(g^{-1}h)^{-1} = h^{-1}g \in S$,  i.e. $h \sim g$. Hence, the corresponding Cayley tournament $\overrightarrow{\textnormal{Cay}}_{\textnormal{tourn}}(\Gamma,S)$ is a tournament.
	\end{proof}
Looking back on the preceding examples, we can now write
\[	Q_k = \textnormal{Cay}(\mathbb{Z}_2^k, S),\ S = \lbrace e_i \mid 1 \leq i \leq k \rbrace \] and
\[ C_n = \textnormal{Cay}(\mathbb{Z}_n,S),\ S = \lbrace \pm 1 \rbrace. \]
\begin{lemma}
\label{cayleytransitive}
	The Cayley digraph $\overrightarrow{\textnormal{Cay}}(\Gamma,S)$ is vertex transitive.
\end{lemma}
\begin{proof}
	For a fixed $w \in \Gamma$ consider the mapping $\rho_w: v \mapsto wv$. It has the inverse $\rho_{w^{-1}}$, hence it is a permutation. Since $(\rho_w(v))^{-1}\rho_w(u) = (wv)^{-1}(wu) = v^{-1}w^{-1}wu = v^{-1}u$, it follows $\rho_w(v) \sim \rho_w(u)$ if and only if $v \sim u$. Thus, $\lbrace \rho_w \mid w \in \Gamma \rbrace$ forms a subgroup of $\textnormal{Aut}(\overrightarrow{\textnormal{Cay}}(\Gamma,S))$. For any $v,w \in \Gamma$, $\rho_{wv^{-1}}(v) = wv^{-1}v = w$, which proves transitivity.
\end{proof}

Note that this is true also for the Cayley graphs and the Cayley tournaments, because they are also Cayley digraphs.

Does a vertex-transitive (di)graph exist, that is not a Cayley (di)graph? The answer is yes. We will give prove to this later, when we have stronger tools available. One then can also ask for which $n \in \mathbb{N}$ precisely, there exists a vertex-transitive graph of order $n$, that is not a Cayley graph. 
 
  Beforehand, let us survey, that every connected vertex-transitive graph is a $\textit{retract}$ of a Cayley graph.

\section{Rectracts}

\begin{definition}
	Let $G$ be a graph and $H \subseteq G$ a subgraph. $H$ is called a $\textit{retract}$, if there exists a homomorphism $\varphi \colon G \rightarrow H$, such that $\varphi \downharpoonright H = \textnormal{id}_H$.
\end{definition}
	
\begin{lemma}
\label{retract}
	Any undirected, connected and vertex-transitive graph is a retract of a Cayley graph.
\end{lemma}
\begin{proof}
	Let $G$ be a connected vertex-transitive graph. For a fixed $v \in V(G)$ consider 
	\[ S := \lbrace g \in \textnormal{Aut}(G) \mid  v \sim v^g \rbrace.\]
	 Because we do not allow cycles, $id \not \in S$. If $g \in S$, $v \sim v^g$, and therefore also $v^{g^{-1}} \sim (v^g)^{g^{-1}} = v$. Since G is undirected, also $v \sim v^{g^{-1}}$ and hence $g^{-1} \in S$. Let 
	 \[ \Gamma := \langle S \rangle \]
	  be the group generated by $S$. By induction, show that $\Gamma$ acts transitively on $V(G)$. Then \autoref{bijec} yields a bijection $\phi$ between the left cosets of $\Gamma_v$ in $\Gamma$ and $V(G)$. We define the map 
	  \[ \varphi \colon \textnormal{Cay}(\Gamma,S) \rightarrow G,\ g \mapsto \phi(g\Gamma_v). \]
	   It holds $S = \Gamma_v S \Gamma_v$. For that, consider $g\in S,\ h,h' \in \Gamma_v$. Then $v = v^h \sim (v^g)^h = v^{hg} = (v^{h'})^{hg} = v^{hgh'}$, therefore $hgh' \in S$. On the other hand, $S \subset \Gamma_v S \Gamma_v$, because $\textnormal{id} \in \Gamma_v$. To verify that $\varphi$ is a homomorphism, consider $ag \in a\Gamma_v$ and $bh \in b\Gamma_v$. Then $ag \sim bh$ if (and only if) $(ag)^{-1}bh = g^{-1}a^{-1}bh \in S$ if (and only if) $a^{-1}b \in S$. The last equivalence follows from the fact that $g \in \Gamma_v$ if and only if $g^{-1} \in \Gamma_v$ and $S = \Gamma_v S \Gamma_v$. Hence, between two distinct left cosets, there is no edge at all or an edge between any two elements in different cosets. Furthermore, between elements of one particular left coset, there is no edge, because $id \not \in S$. We still have to verify that $\phi$ is an homomorphism, i.e. if two left cosets $g\Gamma_v$ and $h\Gamma_v$ of $\Gamma_v$ are completely adjacent, also $w \sim w'$, with $w=\phi(g\Gamma_v), w'=\phi(h\Gamma_v) \in V(G)$. Since $\Gamma$ acts transitively on $V$, we find $a,b \in \Gamma$, such that $w = v^a$ and $w' = v^b$. Hence, $w \sim w'$ if (and only if) $v^a \sim v^b$ if (and only if) $v \sim v^{a^{-1}b}$ if (and only if) $a^{-1}b \in S$. Now use the fact that $g\Gamma_v$ consists exactly of those $a \in \Gamma$, such that $v^a = w$ and $h\Gamma_v$ of those $b\in \Gamma$, such that $v^b= w'$.
	   
	   If we restrict $\varphi$ to representatives of the left cosets, we obtain $id_G$, which proves the claim. 
	\end{proof}
\nocite{hahntardif}

\begin{definition}
	Let $\Gamma$ be a group, $\Gamma_0 \subseteq \Gamma$ a subgroup, and $S \subseteq \Gamma$, such that $1 \not \in S$ and $S=S^{-1}$. The $\textit{Cayley coset graph}$ $\textnormal{Ccg}(\Gamma,S,\Gamma_0)$ is defined by $V(\textnormal{Ccg}(\Gamma,S,\Gamma_0)) = \lbrace a\Gamma_0 \mid a \in \Gamma \rbrace$ and $a\Gamma_0 \sim b\Gamma_0$ if and only if $(ag)^{-1}bh \in S$ for some $g,h \in \Gamma_0$.
\end{definition}

With the conditions of the previous proof, set $\Gamma_0 := \Gamma_v$ to obtain the following Corollary.

\begin{corollary}[Sabidussi \cite{sabidussicoset}]
	Every undirected, connected vertex-\\transitive graph is isomorph to a Cayley coset graph.\qed
\end{corollary}

\section{Recognition}
A permutation group $\Gamma$ acting on a set $V$ is called $\textit{semiregular}$ if the stabilisator is trivial for every $v \in V$. It is called $\textit{regular}$, if it is semiregular and transitive. Recall \autoref{orbitstabilizer} to see that, if $\Gamma$ is semiregular, $\vert v^{\Gamma} \vert = \vert \Gamma \vert$. If $\Gamma$ is also transitive, it exists only one orbit, namely $V$. Hence $\vert V \vert = \vert \Gamma \vert$.\\
A \textit{group isomorphism} is a bijective map $\varphi \colon \Gamma \rightarrow \Gamma'$ between two groups $\Gamma, \Gamma'$, such that $\varphi(gh) = \varphi(g)\varphi(h)\ \textnormal{for all}\ g,h \in \Gamma$. To indicate, that such a map exists, we write $\Gamma \cong \Gamma'$. If $\Gamma = \Gamma'$, such a map is called $\textit{group automorphism}$.

\begin{lemma}
\label{regularsub}
	Let $\overrightarrow{\textnormal{Cay}}(\Gamma,S)$ be a Cayley digraph. There exists a regular subgroup of $\textnormal{Aut}(\overrightarrow{\textnormal{Cay}}(\Gamma,S))$ isomorphic to $\Gamma$.
\end{lemma}
\begin{proof}
	In the proof of \autoref{cayleytransitive}, we have seen that the maps ($\rho_w: \Gamma \rightarrow \Gamma, v \mapsto wv)_{w\in \Gamma}$ form a subgroup of $\textnormal{Aut}(\overrightarrow{\textnormal{Cay}}(\Gamma,S))$ acting transitively on $\Gamma$. The subgroup is isomorphic to $\Gamma$ by $\rho_w \overset{\varphi}{\mapsto} w$, since we have $\varphi(\rho_u\rho_w) = \varphi(\rho_{uw}) = uw = \varphi(\rho_u)\varphi(\rho_w)$ and definition yields that the kernel is trivial and $\varphi$ is surjective. The only element fixing vertices is $\rho_{\mathbbm{1}}=\textnormal{id}_\Gamma$.
\end{proof}

\begin{lemma}
	Let $G$ be an undirected graph and $\Gamma \subseteq \textnormal{Aut}(G)$ a regular subgroup acting on $G$. Then there exists a Cayley graph $\textnormal{Cay}(\Gamma,S)$ isomorphic to $G$. Furthermore, if $G$ is a tournament and $\Gamma \subseteq \textnormal{Aut}(G)$ a regular subgroup acting on $G$, there exists a Cayley tournament $\overrightarrow{\textnormal{Cay}}_{\textnormal{tourn}}(\Gamma,S)$ isomorphic to $G$. 
\end{lemma}
\begin{proof}
	Let $v \in V(G)$ be fixed. For any $w \in V(G)$ we find an unique $\rho_w \in \Gamma$, such that $v^{\rho_w} = w$, since $\Gamma$ is transitive. Consider $S := \lbrace \rho_w \mid v \sim v^{\rho_w} = w \rbrace$. We have $\textnormal{id}_G \not\in S$, because this yields $v \sim v$. Moreover, $S=S^{-1}$, if $G$ is undirected, because $v \sim v^{\rho_w} \ \textnormal{if and only if} \ v^{\rho_w^{-1}} \sim (v^{\rho_w})^{\rho_w^{-1}} = v$	. On the other hand, if $G$ is a tournament, it also holds $v^{\rho_w^{-1}} \sim v\ \textnormal{if and only if}\ v \not\sim v^{\rho_w^{-1}}$. Therefore $S \cap S^{-1}= \emptyset $ and $S \cup S^{-1} = \Gamma \setminus \lbrace 1 \rbrace$. Now let $w,u \in V(G)$. Since $\rho_w^{-1} \in \textnormal{Aut}(G)$, it holds 
	\begin{align*}
		w \sim u \ &\textnormal{if and only if}\  v = (v^{\rho_w})^{\rho_w^{-1}} = w^{\rho_w^{-1}} \sim u^{\rho_w^{-1}} = (v^{\rho_u})^{\rho_w^{-1}} = v^{\rho_w^{-1}\rho_u}\\ &\textnormal{if and only if} \ \rho_w^{-1}\rho_u \in S.
	\end{align*}
	Hence, by identifying every $w \in V(G)$ with $\rho_w \in \Gamma$, we obtain an isomorphism between $G$ and $\textnormal{Cay}(\Gamma,S)$ or $\overrightarrow{\textnormal{Cay}}_{\textnormal{tourn}}(\Gamma,S)$, respectively.
\end{proof}

This yields the result of Sabidussi, a fundamental tool to recognize Cayley graphs.

\begin{corollary}[Sabidussi\cite{sabidussisub}]
\label{sabidussi}
	An undirected graph $G$ is isomorphic to a Cayley graph if and only if $\textnormal{Aut}(G)$ contains a regular subgroup. Moreover, a tournament $G$ is isomorphic to a Cayley tournament if and only if $\textnormal{Aut}(G)$ contains a regular subgroup. \qed
\end{corollary}

We are now able to prove our claim that there exists a vertex-transitive graph which is not a Cayley graph. For that, let $v,k,i \in \mathbb{N}$, such that $v\geq k \geq i$ and $\Omega$ be a fixed set of size v. The graph $J(v,k,i)$ is defined to have the vertex set $V := \lbrace M \in \mathcal{P}(\Omega) \mid |M| = k \rbrace$ and edge set $E := \lbrace \lbrace M, N \rbrace \mid |M \cap N| = i \rbrace$. The graph $J(5,2,0)$ is called $\textit{Petersen graph}$. Its vertices are the $2$-element subsets of $\lbrace 1, ..., 5 \rbrace$. Every $\rho \in S_5$ induces a permutation $\rho'$ on $V(J(5,2,0))$. Furthermore, each $\rho'$ is an automorphism, since $\lbrace a,a' \rbrace, \lbrace b,b' \rbrace \subseteq V(J(5,2,0))$ are distinct if and only if $\rho'(\lbrace a,a' \rbrace) = \lbrace \rho(a), \rho(a') \rbrace$ and $\rho'(\lbrace b,b' \rbrace) = \lbrace \rho(b), \rho(b') \rbrace$ are distinct. This holds, since otherwise $\rho$ would not be injective. Hence, with $\varphi \colon S_5 \mapsto \textnormal{Aut}(J(5,2,0)),\ \rho \mapsto \rho'$, we have $\varphi(S_5) < \textnormal{Aut}(J(5,2,0))$. This yields the vertex-transitivity, since for any $a,a',b,b' \in \lbrace 1,...,5 \rbrace$, pairwise differing, we find an unique $\rho \in S_5$, such that $\rho(a)=b,\ \rho(a')=b'$, i.e. $\rho' \in \varphi(S_5)$, such that $\rho'(\lbrace a,a' \rbrace) = \lbrace b,b' \rbrace$.

Assume the Petersen graph is a Cayley graph. Since the Petersen graph has order 10, \autoref{regularsub} yields, that we find a regular subgroup $\Gamma$ of $\textnormal{Aut}(J(5,2,0))$\\$= S_5$ of order $10$.  But every involution in $\Gamma \subseteq S_5$, that is an element of order $2$, fixes a vertex. This is a contradiction to the regularity of $\Gamma$. Hence, the Petersen graph is a vertex-transitive graph that is not a Cayley graph.

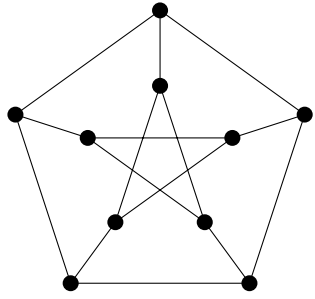
\begin{figure*}
	\begin{center}
\begin{tikzpicture}[]
\draw (18:2cm) -- (90:2cm) -- (162:2cm) -- (234:2cm) -- (306:2cm) -- cycle;
\draw (18:1cm) -- (162:1cm) -- (306:1cm) -- (90:1cm) -- (234:1cm) -- cycle;
\foreach \x in {18,90,162,234,306}{
\draw (\x:1cm) -- (\x:2cm);
\fill (\x:2cm) circle (3pt);
\fill (\x:1cm) circle (3pt);
}
\end{tikzpicture}
\end{center}
\caption{The Petersen graph $J(5,2,0)$.}
\end{figure*}

\section{Isomorphism}
A nearby problem is determining when two Cayley graphs are isomorphic. In this section we will restrict us to the case in which the Cayley graphs are defined on the same group.
\begin{lemma}
\label{autcay}
	Let $\Gamma$ be a group and $\varphi$ a group automorphism of $\Gamma$. Then $\overrightarrow{\textnormal{Cay}}(\Gamma,S)  \cong \overrightarrow{\textnormal{Cay}}(\Gamma,\varphi(S))$. In addition, also $\textnormal{Cay}(\Gamma,S)  \cong \textnormal{Cay}(\Gamma,\varphi(S))$ and $\overrightarrow{\textnormal{Cay}}_{tourn}(\Gamma,S)  \cong \overrightarrow{\textnormal{Cay}}_{tourn}(\Gamma,\varphi(S))$.
\end{lemma}
\begin{proof}
	First, let us verify that $\overrightarrow{\textnormal{Cay}}(\Gamma,\varphi(S))$ is well defined. Since $\varphi$ is an automorphism, if $1 = \varphi(s) \in \varphi(S)$, it follows $S \ni s=1$, which is a contradiction. For the other both cases note that we also have $g \in S$ if and only if $\varphi(g) \in \varphi(S)$ and $g^{-1} \in S$ if and only if $\varphi(S) \ni \varphi(g^{-1}) = \varphi(g)^{-1}$. Now we can use this with $g \in S$ if and only if $g^{-1} \in S$ and $g\in S$ if and only if $g^{-1} \not\in S$, respectively. Finally, let $g,h \in \Gamma$. We obtain $g \sim h \ \textnormal{if and only if} \ g^{-1}h \in S\textnormal{ if and only if} \ \varphi(S) \ni  \varphi(g^{-1}h) = \varphi(g)^{-1}\varphi(h)  \ \textnormal{if and only if} \ \varphi(g) \sim \varphi(h)$. Hence $\varphi$ is an isomorphism in all three cases.
\end{proof}

A group action of a group $\Gamma$ on a set $V$ is said to be \textit{doubly-transitive} or \textit{2-transitive}, if for any $v,v',w,w' \in V,\ v \not= v', w \not = w'$ there exists $\rho \in \Gamma$, such that $v^\rho=w$ and $v'^\rho=w'$.

\begin{lemma}[Burnside \cite{burnside}]
\label{burnsidelem}
	Let $p \in \mathbb{P}$ and $\Gamma$ be a transitive subgroup of $\textnormal{Sym}(p)$. Then either $\Gamma$ is doubly transitive or $\Gamma$ contains a normal Sylow $p$-subgroup of order $p$.
\end{lemma}

Our first deep result is the following Lemma by Turner.

\begin{lemma}[\textnormal{Turner} \cite{turner}]
\label{turnerlemma}
	Let $p \in \mathbb{P}$. Then $\overrightarrow{\textnormal{Cay}}(\mathbb{Z}_p,S) \cong \overrightarrow{\textnormal{Cay}}(\mathbb{Z}_p, S')$ if and only if $S'=aS$ for some $a \in \mathbb{Z}_p^{\times}$. Furthermore, the statement is also well defined for Cayley graphs and Cayley tournaments.
\end{lemma} 
\begin{proof}
	Let $S'=aS$ for some $a \in \mathbb{Z}_p^{\times}$. Since $\varphi_a: \mathbb{Z}_p \rightarrow \mathbb{Z}_p, v \mapsto av$ is a group automorphism, \autoref{autcay} yields $\overrightarrow{\textnormal{Cay}}(\mathbb{Z}_p,S) \cong \textnormal{Cay}(\mathbb{Z}_p,\varphi_a(S)) =  \overrightarrow{\textnormal{Cay}}(\mathbb{Z}_p, S')$ and by the same Lemma it is also well defined for Cayley graphs and Cayley tournaments.
	
	For the other direction it suffices to consider Cayley digraphs. Let $\overrightarrow{\textnormal{Cay}}(\mathbb{Z}_p,S)\\ \overset{\varphi}{\cong} \overrightarrow{\textnormal{Cay}}(\mathbb{Z}_p, S')$. 
	
	Assume $\textnormal{Aut}(\overrightarrow{\textnormal{Cay}}(\mathbb{Z}_p,S))$ or $\textnormal{Aut}(\overrightarrow{\textnormal{Cay}}(\mathbb{Z}_p,S'))$ are doubly-transitive. Then $\overrightarrow{\textnormal{Cay}}(\mathbb{Z}_p,S)$ and $\overrightarrow{\textnormal{Cay}}(\mathbb{Z}_p,S')$ are the same, either $K_p$ or $\overline{K}_p$. To see that each of both are either the one or the other, take without loss of generality $\overrightarrow{\textnormal{Cay}}(\mathbb{Z}_p,S)$ and assume it is not $\overline{K}_p$. Then we have $v \sim v'$ for some $v,v' \in \mathbb{Z}_p,\ v\not=v'$. For any $w,w' \in \mathbb{Z}_p,\ w \not = w'$, we find $g \in \textnormal{Aut}(\overrightarrow{\textnormal{Cay}}(\mathbb{Z}_p,S))$, such that $v^g = w$ and $v'^g = w'$. Hence, $w \sim w'$, which yields $K_p$. Since $K_p$ and $\overline{K}_p$ can not be isomorphic, the claim follows. Therefore, we get $S=S'=\emptyset$ or $S=S'=\lbrace 1,...,p-1 \rbrace$, respectively. Hence set $a := 1 \in (\mathbb{Z}_p)^{\times}$.
	
	 Now assume $\Gamma = \textnormal{Aut}(\overrightarrow{\textnormal{Cay}}(\mathbb{Z}_p,S))$ and $\Gamma' = \textnormal{Aut}(\overrightarrow{\textnormal{Cay}}(\mathbb{Z}_p,S'))$ are not doubly-transitive. We know that Calyey digraphs are vertex-transitive and therefore \autoref{burnsidelem} implies that there exist unique normal Sylow p-subgroups of order $p$ generated by $g = (x_0 \ x_1\	 ...\ x_{p-1})$ and $g'=(y_0 \ y_1\	 ...\ y_{p-1})$. By renaming the vertices we can without loss of generality assume $g = (0\ 1\ ...\ p-1) = g'$. Since 
	 \begin{align*}
	 	\varphi(i) \sim \varphi(j)\ &\textnormal{if and only if}\ i \sim j\\
	 	 &\textnormal{if and only if}\ i^g \sim j^g \\
	 	 &\textnormal{if and only if}\ i+1 \sim j+1\\
	 	 &\textnormal{if and only if}\ \varphi(i+1) \sim \varphi(j+1),
	 \end{align*} it follows that $(\varphi(0) \ \varphi(1)\	 ...\ \varphi(p-1)) \in \textnormal{Aut}(\overrightarrow{\textnormal{Cay}}(\mathbb{Z}_p,S'))$. Hence, 
	 \[(\varphi(0) \ \varphi(1)\	 ...\ \varphi(p-1)) = g'^a = (0 \ a\	 ...\ a(p-1)) \]
	  for some $a \in \mathbb{Z}_p^{\times}$, which we consider to be an element of $\lbrace 1,...,p-1 \rbrace \subset \mathbb{N}$ in $g'^a$. Define $k:=a^{-1}\varphi(0) \in \mathbb{Z}_p$, then $\varphi(0) = ak$, which yields $\varphi(i) = \varphi(0) + ai = ak + ai = a(k+i)\ \textnormal{for all}\ i$. It follows 
	\begin{align*}
		s \in S\ &\textnormal{if and only if}\ i \sim i+s\ \textnormal{for all}\ i\\
		&\textnormal{if and only if}\ \varphi(i) \sim \varphi(i+s)\ \textnormal{for all}\ i\\
		&\textnormal{if and only if}\ a(k+i) \sim  a(k+(i+s)) = a(k+i) + as\ \textnormal{for all}\ i\\
		&\textnormal{if and only if}\ as \in S'.
	\end{align*}
\end{proof}

\begin{theorem}
\label{circulant}
	An undirected graph $G$ with $|G|=p \in \mathbb{P}$ is vertex-transitive if and only if $G$ is isomorphic to a Cayley graph. Furthermore, a tournament $G$ with $|G|=p \in \mathbb{P}$ is vertex-transitive if and only if $G$ is isomorphic to a Cayley tournament.
\end{theorem}
\begin{proof}
	If $G$ is a Cayley graph or a Cayley tournament, it is vertex-transitive because of \autoref{cayleytransitive}.
	
	On the other hand, let $G$ be any undirected, vertex-transitive graph of order $p$. If $\textnormal{Aut}(G)$ is doubly-transitive, we already have seen in the proof of \autoref{turnerlemma} that $G$ is either $K_p$ or $\overline{K}_p$. These are Cayley graphs with connection sets $S=\left \lbrace 1,...,p-1 \right \rbrace$ and $S=\emptyset$, respectively. The automorphism group of a tournament can not be doubly-transitive, since this would yield an undirected graph.
	
	If $\textnormal{Aut}(G)$ is not doubly-transitive, it must contain an unique, normal Sylow $p$-subgroup $\Gamma$ of order $p$ by \autoref{burnsidelem}. Note here, that we must require $\textnormal{Aut}(G)$ to act transitively on $G$, i.e. $G$ to be vertex-transitive. $\Gamma$ must be generated by a permutation of the form $g = (x_0 \ x_1\	 ...\ x_{p-1})$ with $\bigcup_i{x_i} = \lbrace 0, ... ,p-1 \rbrace  $. Without loss of generality. let $x_0$ be chosen such that there exists an edge from it. After renaming $x_i$ to $i$, we can assume $g = (0\ 1\ ...\ p-1)$. Thus, for any $v \in G$ the stabilizer of $v$ is the identity and hence, $\Gamma$ is a regular subgroup. Now by \autoref{sabidussi} the claim follows. Explicitly, let $v \in G$ be fixed, then $G$ is isomorphic to the Cayley graph $\textnormal{Cay}(\Gamma, S)$ with $S := \lbrace \rho_w \mid v \sim v^{\rho_w} = w \rbrace$ by $w \mapsto \rho_w$. If $G$ is a tournament, $\Gamma$ still is a regular subgroup and from \autoref{sabidussi} it also follows, that $G$ is isomorphic to a Cayley tournament.
	\end{proof}

At the end of our proof take for example $v := 0$. Then let $i \in \lbrace 1,...,p-1 \rbrace$, such that $v = 0 \sim i = v^{g_i}$, where $g_{i} = g^i$ in this case.  Now, if $k-l=i\ \textnormal{mod}\ p$, then $g_{k} g_{l}^{-1}=g^k g^{p-l} = g^{k-l}=g_{k-l} = g_{i} \in S$. Hence, $g_{k} \sim g_{l}$, which yields $k \sim l$. Hence $w \sim w+i\ \textnormal{mod}\ p$, for all $w \in \lbrace 0,...,p-1 \rbrace$. We obtained a so called \textit{Circulant} graph.
This already suggests the fact that Circulant graphs can be described in several equivalent ways. One characterization is, that the automorphism group of the graph includes a cyclic subgroup that acts transitively on the graphs vertices. Another characterization is, if some two vertices $v$ and $v+i$ are adjacent, then every two vertices $w$ an $w+i$ are adjacent. A third characterization would be, that the graph is a Cayley graph of a cyclic group.
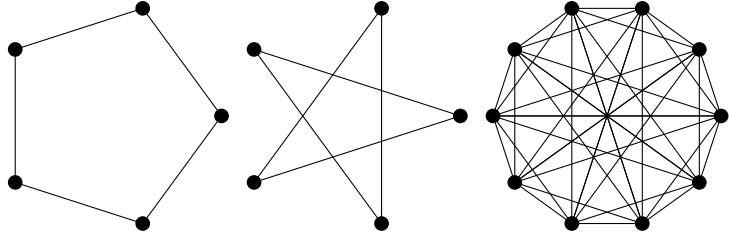
\begin{figure}[H]
\centering
\subfloat{
	\centering
	\begin{tikzpicture}[scale=1]
		\GraphInit[vstyle=Simple]
		\tikzset{VertexStyle/.append style = {color = black, minimum size = 5pt, inner sep = 0pt, outer sep = 0pt}}
		\tikzset{EdgeStyle/.append style = {style = thin}}

		\grCirculant[RA=1.5]{5}{1}
	\end{tikzpicture}
}
\subfloat{
	\centering
	\begin{tikzpicture}[scale=1]
		\GraphInit[vstyle=Simple]
		\tikzset{VertexStyle/.append style = {color = black, minimum size = 5pt, inner sep = 0pt, outer sep = 0pt}}
		\tikzset{EdgeStyle/.append style = {style = thin}}

		\grCirculant[RA=1.5]{5}{2}
	\end{tikzpicture}
}
\subfloat{
	\centering
	\begin{tikzpicture}[scale=1]
		\GraphInit[vstyle=Simple]
		\tikzset{VertexStyle/.append style = {color = black, minimum size = 5pt, inner sep = 0pt, outer sep = 0pt}}
		\tikzset{EdgeStyle/.append style = {style = thin}}

		\grCirculant[RA=1.5]{10}{1,2,4,5}
	\end{tikzpicture}
}
\caption{Three Circulant graphs}
\end{figure}

Be aware that $\textnormal{Cay}(\mathbb{Z}_n,S) \cong \textnormal{Cay}(\mathbb{Z}_n, S')$ if and only if $S'=aS$ for some $a \in \mathbb{Z}_n^{\times}$ does not hold for every $n \in \mathbb{N}$. Consider for example $n = 25$ with 
\[ S = \lbrace 1,4,5,6,9,11,14,16,19,20,21,24 \rbrace \] and
\[ S' = \lbrace 1,4,6,9,10,11,14,15,16,19,21,24 \rbrace. \]
Both corresponding Cayley graphs are wreath-products of 5-cycles, a term we will define in the last chapter, and therefore isomorphic. Though, there exists no $a \in \mathbb{Z}_{25}^{\times}$, such that $S' = aS$. For this, note that it must hold $a=a*	1\in S' \cap \mathbb{Z}_{25}^{\times}$; hence $a \in \lbrace 1,4,6,9,11,14,16,19,21,24 \rbrace$. It follows $a \not= 1$ and $a \not = 4$, because $a*5=4*5=20 \not \in S'$. In the same way all other possibilities yield a contradiction.\\\\

\chapter{Tournaments}
In this chapter we will give the main result of the thesis.
	\section{Vertex-transitive tournaments}
	Let us begin with a short remark. 
\begin{remark}
\label{oddorder}
Every vertex-transitive tournament has odd order.	
\end{remark}
\begin{proof}
	Let $G$ be a vertex-transitive tournament. Vertex-transitivity implies $\textnormal{deg}^+(v) = \textnormal{deg}^+(w) = k$ and $\textnormal{deg}^-(v) = \textnormal{deg}^-(w)=\ell$ for all $v,w \in V(G)$. Hence, we have 
	\[ |G|k = \sum_v \textnormal{deg}^+(v) = \sum_v \textnormal{deg}^-(v) = |G|\ell, \]
	which implies $k = \ell$. Since a tournament is complete, we also have $2k = k+\ell = |G|- 1$, which yields a contradiction, if $|G|$ is even.
\end{proof}
\begin{definition}
	Let $p \in \mathbb{P}\setminus \lbrace 2 \rbrace$ and $p-1 = 2^k r,\ 2 \nmid r$. By recursion is defined
	\begin{align*}
	&S(r) = 0,\\
	&S(m) = \sum_{\substack{m | n | r,\\m < n}}\textnormal{pr}_1(\varphi_p(n))\textnormal{pr}_2(\varphi_p(n)),\ 1 \leq m < r,\\
	&\varphi_p(m) = (k,\frac{p-1}{m}),\ \textnormal{such that}\ 2^{\frac{\frac{p-1}{m}}{2}} - S(m) = k \frac{p-1}{m},\ 1 \leq m \leq r.
	\end{align*}
\end{definition}
\begin{theorem}
	\label{tournamentmain}
	Let $p \in \mathbb{P} \setminus \lbrace 2 \rbrace$ and $p-1 = 2^k r,\ 2 \nmid r$. There exist 
	\[ \sum_{\substack{m | r,\\ 1 \leq m \leq r }} \textnormal{pr}_1(\varphi_p(m)) \]
	isomorphism classes of vertex-transitive tournaments of order $p$.
\end{theorem}
\begin{proof}
	From \autoref{circulant} and the fact that every group of order $p$ is isomorphic to $\mathbb{Z}_p$, we know, that there exists, after identifying isomorphic graphs, an one-to-one correspondence between vertex-transitive tournaments on $p$ vertices and Cayley tournaments on $\mathbb{Z}_p$. 
	We can identify the connection set $S$ of any Cayley tournament with the $\frac{p-1}{2}$- tuple $(x_1,...,x_\frac{p-1}{2})$, where $x_i \in \lbrace i, -i \rbrace$, since $0 \not \in S$ and $i\in S$ or $-i \in S$, but only one holds, for all $i=1,...,\frac{p-1}{2}$. To indicate this identification we write $S \cong_{ident} (x_1,...,x_\frac{p-1}{2})$. It follows, that there exist $2^\frac{p-1}{2}$ possible connection sets $S$ or vertex-transitive tournaments up to isomorphism. From \autoref{turnerlemma} we know that 
	\[ 
	\overrightarrow{\textnormal{Cay}}_{tourn}(\mathbb{Z}_p,S) \cong \overrightarrow{\textnormal{Cay}}_{tourn}(\mathbb{Z}_p,S' )\ \textnormal{if and only if } S'=aS \textnormal{ for some } a \in \mathbb{Z}_p^{\times}.
	\]
	
We define an equivalence relation on the set of connection sets of Cayley tournaments on $\mathbb{Z}_p$ as follows 
\[ S \textnormal{ equivalent to } S'\ \textnormal{if and only if } S'=aS \textnormal{ for some } a \in \mathbb{Z}_p^{\times}. \]
 It follows, that any equivalence class has maximum size $\vline \mathbb{Z}_p^{\times} \vline = p-1$. Assume there is an equivalence class $E$, with $\vline E \vline < p-1$. Let $S \in E$, then for some $a,b \in \mathbb{Z}_p^{\times}$, we have $aS = bS$ or $S = a^{-1}bS$. How can we find all such $S$? \\
	Let $(x'_1,...,x'_\frac{p-1}{2})$ be a $\frac{p-1}{2}$-tuple, such that $x'_i \in \mathbb{Z}_p,\ (\bigcup_i x'_i) \cap (\bigcup_i x'_i)^{-1} = \emptyset$ and $x'_i \not \equiv x'_j$, if $i \not= j$. By reordering and using congruent values in $\lbrace -\frac{p-1}{2},...,\frac{p-1}{2} \rbrace$, we can identify this tuple with $(x_1,...,x_\frac{p-1}{2}),\ x_i \in \lbrace i, -i \rbrace$. We indicate this identification by writing $(x'_1,...,x'_\frac{p-1}{2}) \cong_{rewrite} (x_1,...,x_\frac{p-1}{2})$. Let $a \in \mathbb{Z}_p^{\times}$. Then we obtain \[
	aS \cong_{ident} a(x_1,...,x_\frac{p-1}{2}) = (ax_1,...,ax_\frac{p-1}{2}) \cong_{rewrite} (x'_1,...,x'_\frac{p-1}{2}).\] 
	 Hence, we can use the information on how $a$ acts on $\lbrace -\frac{p-1}{2},...,\frac{p-1}{2} \rbrace$ to understand the multiplication $aS$. We know that $\mathbb{Z}_p^{\times}$ is a disjunct union of left cosets of $\langle a \rangle$: 
	\[ \mathbb{Z}_p^{\times} =\ \stackrel{.} \bigcup_{b \in I\subseteq \mathbb{Z}_p^{\times}} b \langle a \rangle = \stackrel{.} \bigcup_{i=1,...,k} \lbrace y^i_1,...,y^i_{\ell} \rbrace, \] 
	where $I$ is some set of representatives and $\ell = | \langle a \rangle |,\ k=|\mathbb{Z}_p^{\times}/ \langle a \rangle|$ and $y^i_j \in \lbrace -\frac{p-1}{2},...,\frac{p-1}{2} \rbrace$. We will see in a moment, that if there exists some $a$-invariant $S$, $|\langle a \rangle|$ must be odd and hence, $k =|\mathbb{Z}_p^{\times}/ \langle a \rangle|$ must be even. Therefore, the left cosets come in pairs, such that they differ only in sign. We choose one of each, denoted by $\lbrace y'^j_1 ,...,y'^j_{\ell} \rbrace$, and obtain
	\[	a(x_1,...,x_\frac{p-1}{2}) \cong_{rewrite} (x_1,...,x_\frac{p-1}{2}) =: x \]
		if and only if
		\[ x \cong_{ident} \lbrace c_1\lbrace y'^1_1,...,y'^1_{\ell} \rbrace,...,c_{\ell}\lbrace y'^{\frac{k}{2}}_1,...,y'^{\frac{k}{2}}_{\ell} \rbrace \rbrace, \]
		with $c_i \in \lbrace -1,1 \rbrace,\ i=1,...,\ell$. Hence, there are $2^{ \frac{k}{2} }$ possibilities for an $a$-invariant $S$.
		
	 For a better understanding, let us give a short example for $p=11, a=5$. To obtain the following diagram, we start with the most inner two columns,  multiplicate them by $a=5$, reorder them and write them with the corresponding values in $\lbrace -5,...,5 \rbrace$:\\
	 \begin{center}
	 \begin{tabular}{l c c c r}
	 1 & 2 & -3 & -4 & -5\\
	 1 & 2 & 8 & 7 & 6\\
	 \hline
	 6 & 1 & 7 & 2 & 8\\
	 \hline
	 -1 & -2 & -3 & -4 &-5\\
	 1 & 2 & 3 & 4 & 5\\
	 \hline
	 5 & 10 & 4 & 9 & 3\\
	 \hline
	 10 & 9 & 3 & 4 & 5\\
	 -1 & -2 & 3 & 4 & 5\\	 
	 \end{tabular}
	\end{center}
	So the multiplication by $a$ uses the sign of the chosen value in $\lbrace -1, 1 \rbrace$ and produces the value in $\lbrace -5,5 \rbrace$ with the same sign. It takes the sign of the chosen value in $\lbrace -2, 2 \rbrace$ and produces the value in $\lbrace -1,1 \rbrace$ with the different sign, and so on. We therefore have
	 \[ (1 5), (2 -1), (3 4), (4 -2), (5 3)\]
	 or in a cycle way of writing also
	 \[ (1 5 3 4 -2), \]
	 since $(-2 1)$ contains the same information as $(2 -1)$. Note that we only need one half of the diagram. Now we can see which sets are invariant under $a$. These are $\lbrace 1, 5, 3, 4, -2 \rbrace \cong_{ident} (1\ -2\ 3\ 4\ 5)$ and $\lbrace -1, -5, -3, -4, 2 \rbrace \cong_{ident} (-1\ 2\ -3\ -4\ -5)$. We can get rid of the diagram, by looking at the left-cosets of $\langle a \rangle = \langle 5 \rangle = \lbrace 1, 5, 5^2, 5^3, 5^4, 5^5 \rbrace = \lbrace 1, 5, 3, 4, -2 \rbrace \subseteq \mathbb{Z}_p^{\times}$. In this case, there are only two: $\lbrace 1, 5, 3, 4, -2 \rbrace$ and $\lbrace -1, -5, -3, -4, 2 \rbrace$. We always get such related pairs, which both contain the same information. A more sophisticated example would be $p=13, a=3$. We then obtain $\frac{p-1}{|\langle a \rangle|} = \frac{12}{3} = 4$ left-cosets: $\lbrace 1, 3, -4 \rbrace, \lbrace -1, -3, 4 \rbrace, \lbrace 5, 2, 6 \rbrace, \lbrace -5 -2 -6 \rbrace$, and therefore have $2^{\frac{\frac{p-1}{|\langle a \rangle|}}{2}}=2^2=4$ possible $a$-invariant $S$, since we can construct such $S$ by choosing one set of each pair.
	 
	 Let us return to the proof. Which $a$'s do yield an $a$-invariant $S$? Assume $-1 \in \langle a \rangle$, then $1,-1 \in \langle a \rangle$. If $S$ is $a$-invariant, this would imply that the sign of the identified tuple of $S$ in the first place, is the opposite of its sign in the first place, which can not be true. Hence, there is no $a$-invariant S. Since $\langle a \rangle$ is a cyclic group generated by $a$, for any $d \in \mathbb{N}$ with $|\langle a \rangle| = kd$ for some $k \in \mathbb{N}$,  $\langle a^{k} \rangle$ is an unique subgroup of order $d$. Thus, $\lbrace -1, 1 \rbrace \subseteq \langle a \rangle \subseteq \mathbb{Z}_p^{\times}$ is the only subgroup of order $2$ and therefore, if there exists some $a$-invariant $S$, $| \langle a \rangle |$ is not dividible by $2$. 
	 
	 We've seen that we can write any $a$-invariant $S$ as an union of $ \frac{|\mathbb{Z}_p^{\times}/ \langle a \rangle|}{2}$ left cosets of $\langle a \rangle$. This yields the fact that all sets $S'$, which are equivalent to $S$, are $a$-invariant as well. For this, note that since $S' = bS$, for some $b \in \mathbb{Z}_p^{\times}$, $b$ just permutes the left cosets. \\
	 Unfortunately, not all $a$-invariant $S$ are necessarily equivalent.  Assume $b \langle a \rangle=c \langle a \rangle$, then 
	 \begin{align*}
	 	bd\langle a \rangle &= db \langle a \rangle \\
	 	&= dc\langle a \rangle \\
	 	&= cd\langle a \rangle.
	 \end{align*} 
	 Thus, $bS = cS$ and therefore any of the $2^{ \frac{|\mathbb{Z}_p^{\times}/ \langle a \rangle|}{2} }$ $a$-invariant $S$ is contained in an equivalence class of maximum size $|\mathbb{Z}_p^{\times}/ \langle a \rangle|$.
	 
	 Again, assume $\langle a \rangle \subseteq \langle b \rangle \subseteq \mathbb{Z}_p^{\times}$. Since 
	 \[ c \langle b \rangle =\ c\stackrel{.} \bigcup_{d \in I\subseteq \mathbb{Z}_p^{\times}} d \langle a \rangle  = \stackrel{.} \bigcup_{e \in J\subseteq \mathbb{Z}_p^{\times}} e \langle a \rangle, \] 
	 any left coset of $\langle b \rangle$ in $\mathbb{Z}_p^{\times}$ can be written as an union of left cosets of $\langle a \rangle$. The important part here is the following. If $S$ is $b$-invariant, it can be written as union of left cosets of $\langle b \rangle$, where it was chosen one left coset out of each pair differing only in sign. Assume such $S$ contains at least once both left cosets of a pair consisting of only in sign differing left cosets of $\langle a \rangle$. Then $S$ would also contain both elements of the pair of left coset of $\langle b \rangle$, which contain the left coset of $\langle a \rangle$. This clearly is a contradiction to the fact that $S$ is $b$-invariant. Therefore it follows, that if $S$ is $b$-invariant, it also is $a$-invariant. Hence, of the $2^{ \frac{|\mathbb{Z}_p^{\times}/ \langle a \rangle|}{2} }$ $a$-invariant $S$, $2^{ \frac{|\mathbb{Z}_p^{\times}/ \langle b \rangle|}{2} }$ are $b$-invariant, and therefore contained in equivalence classes of maximum size $|\mathbb{Z}_p^{\times}/ \langle b \rangle| <\  |\mathbb{Z}_p^{\times}/ \langle a \rangle|$. 
	 
	 Let $E$ be the equivalence class of a set $S$, which is $a$-invariant, such that $\langle a \rangle$ is the maximal group of all the groups generated by elements, which leave $S$ invariant. Assume that $E$ does not take its maximum size $|\mathbb{Z}_p^{\times}/ \langle a \rangle|$. Then $bS = cS$, for some $b,c \in \mathbb{Z}_p^{\times}$, i.e. $S$ is $d=b^{-1}c$-invariant. If $\langle d \rangle \subseteq \langle a \rangle$, the maximum size will not become smaller. Hence, assume the opposite $\langle a \rangle \subsetneq \langle d \rangle$. But $\langle a \rangle$ is maximal by assumption, so this is a contradiction.  This means that $E$ takes its maximum size $|\mathbb{Z}_p^{\times}/ \langle a \rangle|$. 
	 
	 To obtain the formula, let $1 \leq m \leq r,\ m| r$ and $n > m$, such that $n| r,\ m|n$. We find $a,b \in \mathbb{Z}_p^{\times}$, such that $|\langle a \rangle| = m,\ |\langle b \rangle|=n$ and $\langle a \rangle \subseteq \langle b \rangle$. As we've already seen, it follows that any $b$-invariant $S$ is also $a$-invariant. Hence, of the $2^{ \frac{|\mathbb{Z}_p^{\times}/ \langle a \rangle|}{2} }$ $a$-invariant $S$, $2^{ \frac{|\mathbb{Z}_p^{\times}/ \langle b \rangle|}{2} }$ are $b$-invariant. $S(m)$ computes the sum of sets corresponding to these $n$'s. The number of those sets, which are $a$-invariant, but not $b$-invariant, for any suitable $b$, is then computed by $2^{ \frac{|\mathbb{Z}_p^{\times}/ \langle a \rangle|}{2} } - S(m)$. Hence, the number of equivalence classes is given by
		 $\frac{2^{ \frac{|\mathbb{Z}_p^{\times}/ \langle a \rangle|}{2} } - S(m)}{|\mathbb{Z}_p^{\times}/ \langle a \rangle|}$ = $\textnormal{pr}_1(\varphi(m))$. The total number of equivalence classes is therefore given by the sum over $\textnormal{pr}_1(\varphi(m))$.
	 	 \end{proof}
	  \begin{example}
	  Consider $p=331 \in \mathbb{P}$. In terms of our definitions we would calculate this as follows. We have $p-1=330=2*11*3*5$ and $\varphi_p(\frac{p-1}{2}) = \varphi_p(11*3*5)=(1,2)$, since
\[	2^{\frac{\frac{p-1}{\frac{p-1}{2}}}{2}} - S(\frac{p-1}{2}) = k*(\frac{p-1}{\frac{p-1}{2}}) \Leftrightarrow 2 = k2 \Leftrightarrow k = 1.\]
		Therefore we get $S(11*3)=S(11*5)=S(3*5)=\textnormal{pr}_1(\varphi_p(11*3*5)) \textnormal{pr}_2(\varphi_p(11*3*5)) = 2$. Hence, we obtain $\varphi_p(11*3)=(3,10)$, since
		\[	2^{\frac{\frac{11*3*5*2}{11*3}}{2}} - S(11*3) = k*\frac{11*3*5*2}{11*3} \Leftrightarrow 32 - 2 = k5*2 \Leftrightarrow k = 3.\]
		In the same way we compute $\varphi_p(11*5)=(1,6)$ and $\varphi_p(3*5)=(93,22)$.
		Next, we might want to compute $\varphi_p(11)$. What we already have is $S(11)=1*2+3*10+1*6=38$. Therefore, we compute
		\[ 2^{\frac{\frac{11*3*5*2}{11}}{2}} - S(11) = k*\frac{11*3*5*2}{11} \Leftrightarrow 32768 - 38 = k30 \Leftrightarrow k=1091, \]
		and hence $\varphi_p(11)=(1091,30)$. In the same way one gets 
		\begin{align*}
			\varphi_p(3) &= (327534518354199,110),\\
			\varphi_p(5) &= (130150493,66).
		\end{align*} 
		The last computation is $\varphi_p(1)$. We get $S(1)=1*2 + 3*10 + 1*6 + 93*22 + 1091*30 + 327534518354199 * 110 + 130150493 * 66 = 36028805608929242$. Furthermore, 
		\begin{align*}
			& 2^{\frac{\frac{11*3*5*2}{1}}{2}} - S(1) = k*\frac{11*3*5*2}{1} \\
			& \Leftrightarrow
		46768052394588893382517914646921020600184232445990 = k 11*3*5*2 \\
		& \Leftrightarrow k = 141721370892693616310660347414912183636921916503.
		\end{align*} 
		Hence, \[ \varphi_p(1)=(141721370892693616310660347414912183636921916503, 330). \] 
		By summation we obtain the result of\[141721370892693616310660347414912511171570422384 \] isomorphism classes. Note that the number of maximum-sized equivalence classes is given by $\textnormal{pr}_1(\varphi_p(1))$.
		\end{example}

We can use the following Prolog-program for computations.
\begin{lstlisting}[language=Prolog]
list_sum([X], X).
list_sum([X,Y|T], R) :-
    Z is X+Y,
    list_sum([Z|T], R).

pow(X,Y,Z) :-
    Z is X**Y.

fact(X,0) :-
    Y is X mod 2,
    not(Y == 0),!.
fact(X,C) :-
    Y is X/2,
    fact(Y,D),
    C is D + 1.

div(X,R) :-
    Y is X-1,
    fact(Y,F),
    pow(2,F,D),
    R is (X-1)/D.

s(P,M,0) :-
    div(P,M).

s(P,0,S) :-
    div(P,V),
    findall(E,
	    (between(1,V,N), 
	     0 is V mod N, 
	     varphi(P,N,K,L), 
	     E is K*L),
	    List),
    list_sum(List,S).

s(P,M,S) :-
    div(P,V),
    Mp is M+1,
    findall(E,
	    (between(Mp,V,N),
	     0 is V mod N, 
	     0 is N mod M, 
	     varphi(P,N,K,L), 
	     E is K*L),
	    List),
    list_sum(List,S).

varphi(P,M,K,L) :-
    L is (P-1)/M,
    s(P,M,S),
    pow(2,L/2,E),
    K is (E - S)/L.

result(P,R) :-
    div(P,V),
    findall(K,
	    (between(1,V,M), 
	     0 is V mod M, 
	     varphi(P,M,K,_)),
	    List),
    list_sum(List,R).


\end{lstlisting}
\newpage
The program yields the following results.
\begin{figure*}[h]
	\begin{center}
\begin{tabular}{l | c || l | c}
	p & number of isomorphism classes & p & number of isomorphism classes \\
	\hline
	3 & 1 & 43 & 49940\\
	5 & 1 & 47 & 182362\\
	7 & 2 & 53 & 1290556\\
	11 & 4 & 59 & 9256396\\
	13 & 6 & 61 & 17895736\\
	17 & 16 & 67 & 130150588\\
	19 & 30 & 71 & 490853416\\
	23 & 94 & 73 & 954437292\\
	31 & 1096 & 79 & 7048151672\\
	37 & 7286 & 83 & 26817356776\\
	41 & 26216 & ... &\\
\end{tabular}
\end{center}
\end{figure*}

For $p=11$ the equivalence classes of the corresponding $2^{\frac{p-1}{2}}=32$ generating sets look as follows
\begin{align*}
	&\big \lbrace \lbrace 1, 2, 3, 4, 5 \rbrace, \lbrace 2, 4, 6, 8, 10 \rbrace, \lbrace 3, 6, 9, 1, 4 \rbrace, \lbrace 4, 8, 1, 5, 9 \rbrace, \lbrace 5, 10, 4, 9, 3 \rbrace, \lbrace 6, 1, 7, 2, 8 \rbrace,\\ &\lbrace 7, 3, 10, 6, 2 \rbrace, \lbrace 8, 5, 2, 10, 7 \rbrace, \lbrace 9, 7, 5, 3, 1 \rbrace, \lbrace 10, 9, 8, 7, 6 \rbrace \big \rbrace, \\
	&\big \lbrace \lbrace 1, 2, 3, 7, 5 \rbrace, \lbrace 2, 4, 6, 3, 10 \rbrace, \lbrace 3, 6, 9, 10, 4 \rbrace, \lbrace 4, 8, 1, 6, 9 \rbrace, \lbrace 5, 10, 4, 2, 3 \rbrace, \lbrace 6, 1, 7, 9, 8 \rbrace,\\ &\lbrace 7, 3, 10, 5, 2 \rbrace,\lbrace 8, 5, 2, 1, 7 \rbrace, \lbrace 9, 7, 5, 8, 1 \rbrace, \lbrace 10, 9, 8, 4, 6 \rbrace \big \rbrace,\\
	&\big \lbrace \lbrace 1,2,3,4,6 \rbrace, \lbrace 2,4,6,8,1\rbrace, \lbrace 3,6,9,1,7\rbrace, \lbrace 4,8,1,5,2 \rbrace, \lbrace 5,10,4,9,8 \rbrace, \lbrace 6,1,7,2,3\rbrace,\\ &\lbrace 7,3,10,6,9 \rbrace, \lbrace 8,5,2,10,4 \rbrace, \lbrace 9,7,5,3,10 \rbrace, \lbrace 10,9,8,7,5 \rbrace \big \rbrace,\\
	&\big \lbrace \lbrace 1,9,3,4,5 \rbrace, \lbrace 2, 7, 6, 8, 10 \rbrace \big \rbrace.
\end{align*}
The structure is reflected in $\varphi_{11}(5)=(1,2)$ and $\varphi_{11}(1)=(3,10)$.

\chapter{Conclusion}
\section{Generalisation}
How can we proceed from the previous results? The most urgent question is, whether we can generalize Theorem 4.1.3. to arbitrary $n \in \mathbb{N}$, or at least to $p^k$, with $p \in \mathbb{P}, k \in \mathbb{N}$. Unfortunately, difficulties arise even from the latter case. Right in the beginning of our proof we used the fact, that every vertex-transitive tournament of prime order $p$ can be uniquely identified with a Cayley tournament on $\mathbb{Z}_p$. This was based on \autoref{circulant}, which ultimately used Burnside's \autoref{burnsidelem}. But there is no general version of this statement for prime powers. 

To proceed, we will take a shortcut and focus only on Cayley tournaments, which means, we drop the relation to ordinary tournaments. That this is in general a restriction, we will see in depth by section two. Again, in the proof of our main Theorem 4.1.3. we used \autoref{turnerlemma}, which states 
\[ 
	\overrightarrow{\textnormal{Cay}}_{tourn}(\mathbb{Z}_p,S) \cong \overrightarrow{\textnormal{Cay}}_{tourn}(\mathbb{Z}_p,S' )\ \textnormal{if and only if } S'=aS \textnormal{ for some } a \in \mathbb{Z}_p^{\times}.
	\]
Generalizing this idea, a particular Cayley graph $\textnormal{Cay}(\Gamma,S)$ is called \textit{CI-graph}, if $\textnormal{Cay}(\Gamma,S) \cong \textnormal{Cay}(\Gamma, S')$ if and only if $S'=aS$ for some group automorphism $a \in \textnormal{Aut}(\Gamma)$.
	Furthermore, $\Gamma$ is called \textit{CI-group}, if for arbitrary $S,S', \textnormal{Cay}(\Gamma,S) \cong \textnormal{Cay}(\Gamma, S')$ if and only if $S'=aS$ for some group automorphism $a \in \textnormal{Aut}(\Gamma)$. Hence, $\Gamma$ is a CI-group if and only if every Cayley graph on $\Gamma$ is a CI-graph.
In the same way we define the $\textit{DCI-digraph}\ \overrightarrow{\textnormal{Cay}}(\Gamma,S)$ and the $\textit{DCI-group}$. 
In other words, the above statement of \autoref{turnerlemma} states that $\overrightarrow{\textnormal{Cay}}_{tourn}(\mathbb{Z}_p,S)$ is a DCI-digraph for reasonable $S$. Even though, from this only, we can not deduce that $\mathbb{Z}_p$ is a DCI-group, since there are sets $S$, which result in a Cayley digraph, but not in a Cayley tournament.

Though, the task of determining all cyclic CI-groups and DCI-groups has been completed by Muzychuk; his results go as follows.
\begin{lemma}[Muzychuk \cite{muzychuk2}]
	The cyclic group $\mathbb{Z}_n$ is a CI-group if and only if $n=8, 9, 18$ or $n=2^{e}m$, where m is odd and square-free and $e \in \lbrace 0,1,2\rbrace$.
\end{lemma}
\begin{lemma}[Muzychuk \cite{muzychuk2}]
	The cyclic group $\mathbb{Z}_n$ is a DCI-group if and only if $n=2^{e}m$, where m is odd and square-free and $e \in \lbrace 0,1,2\rbrace$.
	\end{lemma}
Following our main proof, we could try to deduce the number of isomorphy classes of Cayley tournaments on $\mathbb{Z}_n$ with $n=2^e \cdot p_1\cdot...\cdot p_k$, where $p_i \in \mathbb{P}\setminus \lbrace 2 \rbrace, p_i \not= p_j$, if $i\not=j$ and $e = 0$ (every vertex-transitive tournament has odd order). Still, by going on with the construction of our proof, we find that we used the fact that $\mathbb{Z}_p^{\times}$ is cyclic to obtain exactly one subgroup for any given suitable order.
Relating to this problem, one finds the following fact.
\begin{remark}
	$\mathbb{Z}_n^{\times}$ is cyclic if and only if $n = 2,4,p^k,2p^k$, where $p \in \mathbb{P}\setminus \lbrace 2\rbrace, k \in \mathbb{N}$. 
\end{remark} 
Hence, we are stuck, since our chosen $n$ either results in $\mathbb{Z}_n$ being a DCI-group, or in $\mathbb{Z}_n^{\times}$ being cyclic, and both holds only in the trivial case of $n=p \in \mathbb{P}\setminus \lbrace 2\rbrace$.

\section{Non-Cayley vertex-transitive tournaments}
As said before, handling Cayley tournaments, instead of vertex-transitive tournaments is actually a restriction. Precisely, there exists a vertex-transitive tournament, that is not a Cayley tournament. Even more, one can construct infinite families of such graphs. How is this done? An import ingredient are metacirculant digraphs, which were defined by Alspach and Parsons \cite{alspach1982construction} as follows. Let $m \in \mathbb{N}_{\geq 1}, n \in \mathbb{N}_{\geq 2}, a \in \mathbb{Z}_n^{\times}$ and $S_0,...,S_{m-1} \subseteq \mathbb{Z}_n$, such that
\begin{enumerate} 
\item $0 \not \in S_0$
\item $a^mS_r = S_r$ for $0 \leq r \leq m-1$.	
\end{enumerate}
The \textit{metacirculant digraph} $G=G(m,n,a,S_0,...,S_{m-1})$ is the digraph with vertex set
\[\lbrace v_j^i \mid i \in \mathbb{Z}_m, j \in \mathbb{Z}_n \rbrace \]
and edge set 
\[ \lbrace \lbrack v_j^i, v_h^{i+r} \rbrack \mid r \in \mathbb{Z}_m, h \in j+a^i S_r \rbrace. \]
We define permutations $\rho, \sigma$ on $V(G)$ by $(v_j^i)^{\rho} = v_{j+1}^i, (v_j^i)^{\sigma} = v_{aj}^{i+1}$ and obtain the following characterization.

\begin{lemma}[Alspach and Parson \cite{alspach1982construction}]
\label{alspach}
	The metacirculant $G=G(m,n,a,S_0,\\...,S_{m-1})$ is vertex-transitive with $\langle \rho, \sigma \rangle \leq \textnormal{Aut}(G)$. Conversely, any digraph $G'$ with vertex set $V(G)$ and $\langle \rho, \sigma \rangle \leq \textnormal{Aut}(G)$ is an $(m,n)$-metacirculant.
\end{lemma}
Maru{\v{s}}i{\v{c}} \cite{marusic} has proven that every vertex-transitive digraph of order $p^k$, with $k \leq 3$ is necessarily a Cayley digraph. On the other hand, he also constructed for each $k \geq 4$ a non-Cayley vertex-transitive digraph of order $p^k$:\begin{lemma}[Maru{\v{s}}i{\v{c}} \cite{marusic}]
\label{marusicmeta}
	Let $p\in \mathbb{P}, k \geq 3, a \in \mathbb{Z}_{p^k}^{\times}$ of order $p^2$. Let $S_0 = \langle a^p\rangle$ and $S_1 = \langle 0 \rangle$. Let $S_i = \emptyset, 2 \leq i \leq p-1, p \not=2$. Then the metacirculant digraph $G(p^k,p,a)=G(p^k,p,S_0,...,S_{p-1})$ is not a Cayley digraph.
\end{lemma}
Based on these results, Mansilla \cite{mansilla2004infinite} then gave a construction on an infinite family of non-Cayley vertex-transitive tournaments of order $p^k, p\in \mathbb{P}\setminus \lbrace 2 \rbrace, k \geq 4$. Therefore, let $a \in \mathbb{Z}_{p^k}^{\times}$ be of order $p^2$ and $G=G(p^k,p,a)$ the corresponding metacirculant digraph in \autoref{marusicmeta}. Let $\Gamma$ be the Sylow $p$-subgroup of $\textnormal{Aut}(G)$ containing $\rho$ and $\sigma$. The main idea now is to work out a tournament $T=T(p^k,p,a)$ with $V(T)=V(G)$ and $\textnormal{Aut}(T) = \Gamma$. In the original proof of \autoref{marusicmeta} one sees that $\Gamma$ does not have any regular subgroup, and hence $T(p^k,p,a)$ will not be Cayley. Though, since $\langle \rho, \sigma \rangle \leq \Gamma =\textnormal{Aut}(T)$, by \autoref{alspach} it follows that $T$ will be a $(p^k,p)$-metacirculant digraph, wherefore it is vertex-transitive.  

How can we from here obtain infinite many non-Cayley vertex-transitive tournaments? Let $q \in \mathbb{P}, q \equiv 3\ \textnormal{mod}\ 4$. The \textit{Paley tournament} $P(q)$ has vertex set $\mathbb{F}_{p^n}$ and $r \sim s$ if and only if $s-r$ is a non-zero square. The lexicographic product of two digraphs $G_1, G_2$ has vertex-set $V(G_1) \times V(G_2)$ and $(v,w) \sim (v',w')$ if and only if $v \sim v'$ or $v=v'$ and $w \sim w'$. Mansilla proves the following result by defining new non-Cayley vertex-transitive tournaments from old ones, using lexicographic product with Paley tournaments. 
\begin{lemma}
	Let $m=p_1^{\ell_1}\cdot ... \cdot p_r^{\ell_r},\ 2 < p_1 < ... < p_r,\ \ell_1,...,\ell_r \in \mathbb{N}_{\geq 1}$. If there exists an index $i, 1 \leq i \leq r$, such that $l_i \geq 4$ and $p_j \equiv 3\ \textnormal{mod}\ 4$ for every $1 \leq j \leq r, i\not=j$, then there is a non-Cayley vertex-transitive tournament of order $m$. 
\end{lemma}
Therefore, if we have a non-Cayley vertex-transitive tournament of order $s=p^k, k \geq 4$, we obtain a new non-Cayley vertex-transitive tournament of order $rs$, for any $r$ product of primes congruent $3$ mod $4$.

\section{Cayley graphs on non-isomorphic groups}
One of the problems that arise when researching Cayley tournaments with non-prime order, is the fact that one might end up with multiple non-isomorphic groups of the same order. It is then straightforward to show that there exist non-isomorphic Cayley tournaments on non-isomorphic groups with same order. But the interesting question rather is, if there exists a Cayley tournament on a certain group one, that is not isomorphic to any other Cayley tournament on a non-isomorphic group two with the same order at all. There has been some research by Anne Joseph \cite{joseph1995isomorphism} concerning squares of primes. Here we will see a short survey of the extended version by Joy Morris\cite{morris1999isomorphic}. But beforehand, let us verify the claim made first. 

Consider $n=3^2$, then up to isomorphy the only non-isomorphic groups are $\mathbb{Z}_9$ and $\mathbb{Z}_3^2$. We define the sets 
\[	S=\lbrace 1, 7, 3, 5 \rbrace \subseteq \mathbb{Z}_9 \] and \[S'=\lbrace (0,1), (2,0), (1,1), (2,1) \rbrace \subseteq \mathbb{Z}_3^2. \]
Both sets yield Cayley tournaments of order $9$. We will now look deeper into the number of triangles with every edge pointing in the same direction. Remember that $a \sim b$ if and only if $b=a+s, s\in S$ and $(a,b) \sim (a',b')$ if and only if $(a',b') = (a,b) + s, s \in S'$.\begin{figure}[h]
\centering
	\begin{tikzpicture}[scale=2]
    \node (A) at (1,1)  {a+8=a-1};
    \node (B) at (1,0.1)  {a+2=a-7};
    \node (C) at (1,-0.8)  {a+6=a-3};
    \node (D) at (1,-1.8)  {a+4=a-5};
    \node (E) at (2,-0.3)  {a};
    \node (F) at (3,1)  {a+1};
    \node (G) at (3,0.1)  {a+7};
    \node (H) at (3,-0.8)  {a+3};
    \node (I) at (3,-1.8)  {a+5};
    \node (J) at (4,1.5)  {a+2};
    \node (K) at (4,1.25) {a+8};
    \node (L) at (4,1)  {a+4};
    \node (M) at (4,0.75)  {a+6};
    \node (N) at (4,0.5)  {a+8};
    \node (O) at (4,0.25)  {a+5};
    \node (P) at (4,0)  {a+1};
    \node (Q) at (4,-0.25)  {a+3};
    \node (R) at (4,-0.5) {a+4};
    \node (S) at (4,-0.75)  {a+1};
    \node (T) at (4,-1)  {a+6};
    \node (U) at (4,-1.25)  {a+8};
    \node (V) at (4,-1.5)  {a+6};
    \node (W) at (4,-1.75)  {a+3};
    \node (X) at (4,-2)  {a+8};
    \node (Y) at (4,-2.25)  {a+1};

    \draw[->] (A) to (E);
    \draw[->] (B) to (E);
    \draw[->] (C) to (E);
    \draw[->] (D) to (E);
    \draw[->] (E) to (F);
    \draw[->] (E) to (G);
    \draw[->] (E) to (H);
    \draw[->] (E) to (I);
    \draw[->] (F) to (J);  
    \draw[->] (F) to (K);  
    \draw[->] (F) to (L);  
    \draw[->] (F) to (M);  
    \draw[->] (G) to (N);
    \draw[->] (G) to (O);  
    \draw[->] (G) to (P);  
    \draw[->] (G) to (Q); 
    \draw[->] (H) to (R);  
    \draw[->] (H) to (S);  
    \draw[->] (H) to (T);  
    \draw[->] (H) to (U);  
    \draw[->] (I) to (V);  
    \draw[->] (I) to (W);  
    \draw[->] (I) to (X);  
    \draw[->] (I) to (Y);  
   
\end{tikzpicture}
\caption{$S=\lbrace 1, 7, 3, 5 \rbrace \subseteq \mathbb{Z}_9$}
\end{figure}

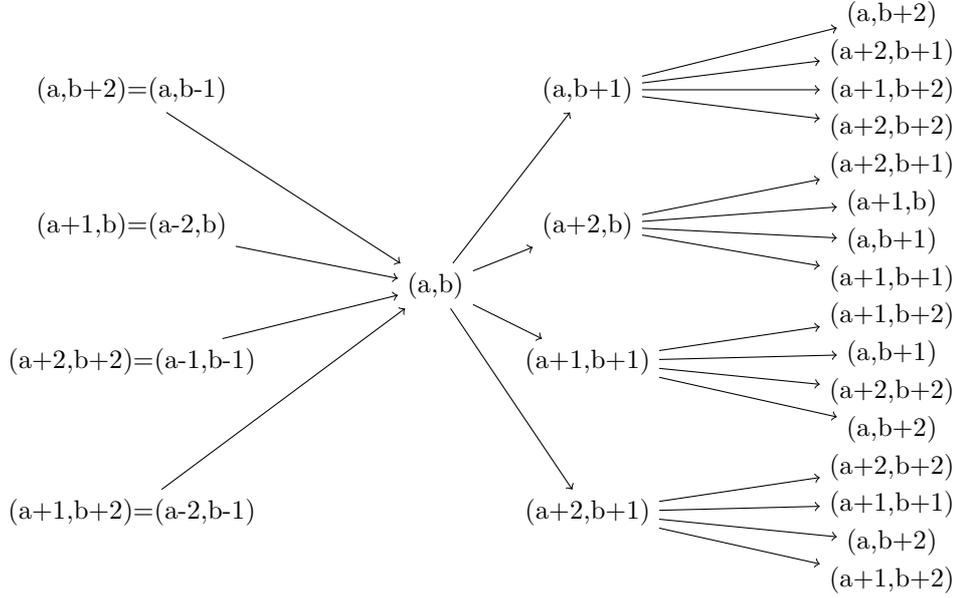
\begin{figure}[h]
\centering
	\begin{tikzpicture}[scale=2]
    \node (A) at (1,1) {(a,b+2)=(a,b-1)};
    \node (B) at (1,0.1)  {(a+1,b)=(a-2,b)};
    \node (C) at (1,-0.8)  {(a+2,b+2)=(a-1,b-1)};
    \node (D) at (1,-1.8)  {(a+1,b+2)=(a-2,b-1)};
    \node (E) at (3,-0.3)  {(a,b)};
    \node (F) at (4,1)  {(a,b+1)};
    \node (G) at (4,0.1)  {(a+2,b)};
    \node (H) at (4,-0.8)  {(a+1,b+1)};
    \node (I) at (4,-1.8)  {(a+2,b+1)};
    \node (J) at (6,1.5)  {(a,b+2)};
    \node (K) at (6,1.25)  {(a+2,b+1)};
    \node (L) at (6,1)  {(a+1,b+2)};
    \node (M) at (6,0.75)  {(a+2,b+2)};
    \node (N) at (6,0.5)  {(a+2,b+1)};
    \node (O) at (6,0.25)  {(a+1,b)};
    \node (P) at (6,0)  {(a,b+1)};
    \node (Q) at (6,-0.25)  {(a+1,b+1)};
    \node (R) at (6,-0.5)  {(a+1,b+2)};
    \node (S) at (6,-0.75)  {(a,b+1)};
    \node (T) at (6,-1)  {(a+2,b+2)};
    \node (U) at (6,-1.25)  {(a,b+2)};
    \node (V) at (6,-1.5)  {(a+2,b+2)};
    \node (W) at (6,-1.75)  {(a+1,b+1)};
    \node (X) at (6,-2) {(a,b+2)};
    \node (Y) at (6,-2.25) {(a+1,b+2)};
    \draw[->] (A) to (E);
    \draw[->] (B) to (E);
    \draw[->] (C) to (E);
    \draw[->] (D) to (E);
    \draw[->] (E) to (F);
    \draw[->] (E) to (G);
    \draw[->] (E) to (H);
    \draw[->] (E) to (I);
    \draw[->] (F) to (J);  
    \draw[->] (F) to (K);  
    \draw[->] (F) to (L);  
    \draw[->] (F) to (M);  
    \draw[->] (G) to (N);
    \draw[->] (G) to (O);  
    \draw[->] (G) to (P);  
    \draw[->] (G) to (Q); 
    \draw[->] (H) to (R);  
    \draw[->] (H) to (S);  
    \draw[->] (H) to (T);  
    \draw[->] (H) to (U);  
    \draw[->] (I) to (V);  
    \draw[->] (I) to (W);  
    \draw[->] (I) to (X);  
    \draw[->] (I) to (Y);  
   
\end{tikzpicture}
\caption{$S'=\lbrace (0,1), (2,0), (1,1), (2,1) \rbrace \subseteq \mathbb{Z}_3^2$}
\end{figure}

In Figure 5.1 take for example $a=0$ and $a+1=1$. They have in common $4$ triangles with every edge pointing in the same direction. On the other hand, in Figure 5.2, we can not find two such vertices. This means $\overrightarrow{\textnormal{Cay}}_{tourn}(\mathbb{Z}_9,S) \not \cong \overrightarrow{\textnormal{Cay}}_{tourn}(\mathbb{Z}_3^2,S')$, which proves the claim.
We will now engage in the results of Joy Morris \cite{morris1999isomorphic}.

Let $G,H$ be digraphs. The \textit{wreath product} $G \wr H$ is defined to be the digraph with $V(G \wr H) = V(G) \times V(H)$ and $(v,w) \sim (v',w')$ if and only if 
\[ v=v'\ \textnormal{and}\ w \sim w' \] or
	\[ v \sim v'. \]
	
A partial order on the set of abelian groups of order $p^k, p \in \mathbb{P}$ is defined as follows. Say $\Gamma' \leq_{po} \Gamma$ if and only if there exists a chain of subgroups 
\[ \Gamma_1 < \Gamma_2 < ... < \Gamma_n=\Gamma, \] 
such that $\Gamma_1, (\Gamma_2 / \Gamma_1),...., (\Gamma_n / \Gamma_{n-1})$ are cyclic and
	\[ \Gamma' \cong \Gamma_1 \times (\Gamma_2 / \Gamma_1) \times ....\times (\Gamma_n / \Gamma_{n-1}). \]
We give a short example for $k=5$. 
\begin{example}
	From the fundamental theorem of finitely generated abelian groups it follows, that any abelian group of order $p^5$ is isomorphic to a group of the form $\mathbb{Z}_{p^{k_1}} \times ... \times \mathbb{Z}_{p^{k_{\ell}}}, \sum_i k_i =5$.
For $m \geq n$ the inclusion $\mathbb{Z}_{p^m} \overset{\iota}{\hookrightarrow} \mathbb{Z}_{p^n}, \lbrack x \rbrack_m \mapsto \lbrack x \rbrack_n$ is well-defined: $\lbrack x \rbrack_m = \lbrack y \rbrack_m$ equals $x = y + lp^m = y + (lp^{m-n})p^n, l \in \mathbb{Z}$, which equals $\lbrack x \rbrack_n = \lbrack y \rbrack_n$. Hence, we have $\mathbb{Z}_{p^n} \cong \mathbb{Z}_{p^m} / \textnormal{ker}(\iota) \cong \mathbb{Z}_{p^m} / \mathbb{Z}_{p^{m-n}}$. Now consider $\Gamma' = \mathbb{Z}_p \times \mathbb{Z}_{p^4} $ and $\Gamma = \mathbb{Z}_{p^5}$. Since $\Gamma$ is additive-cyclic, for every $k=1,...,5$ there exists a unique subgroup of order $p^k$, which in this case is $\mathbb{Z}_{p^k}$. Hence, we have a $1$-chain 
\[ \Gamma_1 = \mathbb{Z}_p < \mathbb{Z}_{p^5} = \Gamma \]
 and 
 \[ \Gamma' = \mathbb{Z}_p \times \mathbb{Z}_{p^4} \cong \mathbb{Z}_p \times (\mathbb{Z}_{p^5} / \mathbb{Z}_p) = \Gamma_1 \times (\Gamma / \Gamma_1), \] 
 which yields $\Gamma' \leq_{po} \Gamma$. Next consider $\Gamma' = \mathbb{Z}_p \times \mathbb{Z}_p \times \mathbb{Z}_{p^3} $ and $\Gamma = \mathbb{Z}_p \times \mathbb{Z}_{p^4}$. We again obtain a $1$-chain
 \[ \Gamma_1 = \lbrace 0 \rbrace \times \mathbb{Z}_p < \mathbb{Z}_p \times \mathbb{Z}_{p^4} = \Gamma, \] and
  \begin{align*}
  	\Gamma' = \mathbb{Z}_p \times \mathbb{Z}_p \times \mathbb{Z}_{p^3} & \cong \lbrace 0 \rbrace \times \mathbb{Z}_p \times \mathbb{Z}_p \times \mathbb{Z}_{p^3} \\ 
  	& \cong \lbrace 0 \rbrace \times \mathbb{Z}_p \times \big((\mathbb{Z}_p \times \mathbb{Z}_{p^4})/( \lbrace 0 \rbrace \times \mathbb{Z}_p)\big) = \Gamma_1 \times (\Gamma / \Gamma_1).
  \end{align*}
   Going on, this yields Figure 5.3.
\end{example}

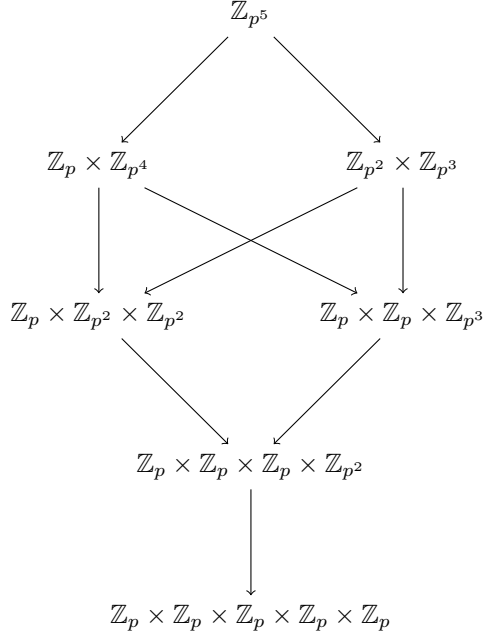
\begin{figure}[h]
\centering
	\begin{tikzpicture}[scale=2]
    \node (A) at (1,4)  {$\mathbb{Z}_{p^5}$};
    \node (B) at (0,3)  {$\mathbb{Z}_{p} \times \mathbb{Z}_{p^4}$};
    \node (C) at (2,3)  {$\mathbb{Z}_{p^2} \times \mathbb{Z}_{p^3}$};
    \node (D) at (0,2)  {$\mathbb{Z}_{p} \times \mathbb{Z}_{p^2} \times \mathbb{Z}_{p^2}$};
    \node (E) at (2,2)  {$\mathbb{Z}_{p} \times \mathbb{Z}_{p} \times \mathbb{Z}_{p^3}$};
    \node (F) at (1,1)  {$\mathbb{Z}_{p} \times \mathbb{Z}_{p} \times \mathbb{Z}_{p} \times \mathbb{Z}_{p^2}$};
    \node (G) at (1,0)  {$\mathbb{Z}_{p} \times \mathbb{Z}_{p} \times \mathbb{Z}_{p} \times \mathbb{Z}_{p} \times \mathbb{Z}_{p}$};

    \draw[->] (A) to (B);
    \draw[->] (A) to (C);
    \draw[->] (B) to (E);
    \draw[->] (C) to (D);
    \draw[->] (B) to (D);
    \draw[->] (C) to (E);
    \draw[->] (D) to (F);
    \draw[->] (E) to (F);
    \draw[->] (F) to (G);

\end{tikzpicture}
\caption{Partial order for abelian groups of order $p^5$.}
\end{figure}

 We can give the main result now.

\begin{theorem}
	Let $\Gamma$ be an abelian group of order $p^k, p \in \mathbb{P}\setminus\lbrace 2 \rbrace$ and $G:=\overrightarrow{\textnormal{Cay}}(\Gamma,S)$. Then the following are equivalent:
	\begin{enumerate}
		\item $G$ is isomorphic to a Cayley digraph on both $\mathbb{Z}_{p^k}$ and $\Gamma'$, where $\Gamma'$ is an arbitrary abelian group of order $p^k$.
	
			\item There exists a chain of subgroups $\Gamma_1 < \Gamma_2 < ... < \Gamma_n=\Gamma$, such that 
					\begin{enumerate}
			\item $\Gamma_1, (\Gamma_2 / \Gamma_1),...., (\Gamma_n / \Gamma_{n-1})$ are cyclic,
			\item $\Gamma_1 \times (\Gamma_2 / \Gamma_1) \times ....\times (\Gamma_n / \Gamma_{n-1}) \leq_{po} \Gamma',$
			\item For all $s \in S\setminus \Gamma_i$, we have $s + \Gamma_i \subseteq S$, for $i=1,...,n-1$
		\end{enumerate}
		\item There exist Cayley digraphs $G_1,...,G_n$ on cyclic $p$-groups $\Gamma_1,...,\Gamma_n$, such that $\Gamma_1 \times ... \times \Gamma_n \leq_{po} \Gamma$ and $G \cong G_n \wr ... \wr G_1$.\\\\
		These in turn imply:
		\item $G$ is isomorphic to Cayley digraphs on every abelian group of order $p^k$, that is greater than $\Gamma'$ in the partial order.
	\end{enumerate}
\end{theorem}

\begin{example}Reconsider $\overrightarrow{\textnormal{Cay}}_{tourn}(\mathbb{Z}_9,\lbrace 1, 7, 3, 5 \rbrace)$. We've already seen that it is not isomorphic to $\overrightarrow{\textnormal{Cay}}_{tourn}(\mathbb{Z}_3^2,\lbrace (0,1), (2,0), (1,1), (2,1) \rbrace)$. Assume there exists a Cayley tournament on $\Gamma' = \mathbb{Z}_3^2$, that it is isomorphic to. Then the first statement holds, since every Cayley tournament is a Cayley digraph, and hence, the second statement must hold as well. There exist two subgroups of $\Gamma = \mathbb{Z}_9$, namely $\Gamma_1 = \langle 3 \rangle = \lbrace 0, 3, 6 \rbrace$ and $\lbrace 0 \rbrace$. Without loss of generality we can consider only $\Gamma_1$. Now $2(c)$ yields $\textnormal{for all}\ s \in S \setminus \Gamma_1 = \lbrace 1, 7, 3, 5 \rbrace \setminus \lbrace 0, 3, 6 \rbrace = \lbrace 1, 7, 5 \rbrace$: 
\[ s + \Gamma_1 = s + \lbrace 1, 7, 3, 5 \rbrace \subseteq S = \lbrace 1, 7, 3, 5 \rbrace. \]
This is not the case for $s=1$, which is a contradiction. Hence, \\$\overrightarrow{\textnormal{Cay}}_{tourn}(\mathbb{Z}_9,\lbrace 1, 7, 3, 5 \rbrace)$ is not isomorphic to any Cayley tournament on $\mathbb{Z}_3^2$.
\end{example}
\begin{example}
	On the other hand, showing that \\$\overrightarrow{\textnormal{Cay}}_{tourn}(\mathbb{Z}_3^2,\lbrace (0,1), (2,0), (1,1), (2,1) \rbrace)$ is not isomorphic to any Cayley tournament on $\mathbb{Z}_9$ does not work as smoothly as before. Assume the opposite, then $1.$ is true with $\Gamma' = \mathbb{Z}_9$. Again, without loss of generality it suffices to investigate the non-trivial subgroups, since the chain of subgroups consisting only of $\Gamma = \mathbb{Z}_3^2$ and the trivial subgroup will not satisfy $2(a)$. There are four subgroups of order three: 
	\begin{enumerate}
		\item $\Gamma_1 = \lbrace (0,0), (1,0), (2,0) \rbrace,$
		\item $\Gamma_2 = \lbrace (0,0), (0,1), (0,2) \rbrace,$
		\item $\Gamma_3 = \lbrace (0,0), (1,1), (2,2) \rbrace,$
		\item $\Gamma_4 = \lbrace (0,0), (1,2), (2,1) \rbrace.$
	\end{enumerate}
	Hence, with $S = \lbrace (0,1), (2,0), (1,1), (2,1) \rbrace$ we obtain :
		\begin{enumerate}
		\item $S \setminus \Gamma_1 = \lbrace (0,1), (1,1), (2,1) \rbrace,$
		\item $S \setminus \Gamma_2 = \lbrace (2,0), (1,1), (2,1) \rbrace,$
		\item $S \setminus \Gamma_3 = \lbrace (0,1), (2,0), (2,1) \rbrace,$
		\item $S \setminus \Gamma_4 = \lbrace (0,1), (2,0), (1,1) \rbrace.$
	\end{enumerate}
	Now in $2.$ we have $(2,0) + (0,2) = (2,2) \not\in S$, in 3. $(0,1) + (1,1) = (1,2) \not \in S$, and in 4. $(0,1) + (2,1) = (2,2) \not \in S$. But in 1. we will unfortunately succeed. So our proof by contradiction does not work. Even more, the $1$-chain
	\[ \Gamma_1 = \lbrace (0,0), (1,0), (2,0) \rbrace < \mathbb{Z}_3^2 = \Gamma \] actually fulfills $2.$: 
	On the one hand we have $\Gamma_1 = \lbrace (0,0), (1,0), (2,0) \rbrace$ being cyclic with $(1,0)$ and on the other hand we have 
 \begin{align*}
 	\Gamma / \Gamma_1 &= \mathbb{Z}_3^2 / \lbrace (0,0), (1,0), (2,0) \rbrace\\ &= \big\lbrace \lbrace (0,0), (1,0), (2,0) \rbrace, \lbrace (1,1),(2,1),(0,1) \rbrace, \lbrace (2,2), (0,2), (1,2) \rbrace \big \rbrace,
 \end{align*}
 which is cyclic with $\lbrace (1,1),(2,1),(0,1) \rbrace$. 
Furthermore, we have 
	\begin{align*}
		\Gamma_1 \times \Gamma / \Gamma_1 &= \lbrace (0,0), (1,0), (2,0) \rbrace \times (\mathbb{Z}_3^2 / \lbrace (0,0), (1,0), (2,0) \rbrace)\\ &\cong \mathbb{Z}_3^2 \\&= \mathbb{Z}_3 \times (\mathbb{Z}_9 / \mathbb{Z}_3)\\&= \Gamma'_1 \times (\Gamma' / \Gamma'_1)
	\end{align*}
by mapping $((i,0),\lbrack (j,j) \rbrack) \mapsto (i,j)$.
Hence, $\Gamma_1 \times (\Gamma / \Gamma_1) \leq_{po} \Gamma'$. It follows that, $\overrightarrow{\textnormal{Cay}}_{tourn}(\mathbb{Z}_3^2,\lbrace (0,1), (2,0), (1,1), (2,1) \rbrace)$ is isomorphic to a Cayley digraph on $\mathbb{Z}_9$. Due to isomorphy, this Cayley digraph is also a tournament. But therefore it must be finally a Cayley tournament. 
	
\end{example}

\chapter*{Statement}
Die vorliegende Arbeit habe ich selbständig verfasst und keine anderen als die angegebenen Hilfsmittel - insbesondere keine im Quellenverzeichnis nicht benannten Internet-Quellen - benutzt. Die Arbeit habe ich vorher nicht in einem anderen Prüfungsverfahren eingereicht. Die eingereichte schriftliche Fassung entspricht genau der auf dem elektronischen Speichermedium.
 \vskip 2cm \noindent Stefan Zetzsche

\bibliography{verzeichnis1}

\begin{thebibliography}{10}

\bibitem{alspachmishna}
Brian Alspach and Marni Mishna.
\newblock Enumeration of cayley graphs and digraphs.
\newblock {\em Discrete mathematics}, 256(3):527--539, 2002.

\bibitem{alspach1982construction}
Brian Alspach and TD~Parsons.
\newblock A construction for vertex-transitive graphs.
\newblock {\em Canad. J. Math}, 34(2):307--318, 1982.

\bibitem{burnside}
W~Burnside.
\newblock On some properties of groups of odd order.
\newblock {\em Proceedings of the London Mathematical Society}, 1(1):162--184, 1900.

\bibitem{godsil2013algebraic}
Chris Godsil and Gordon~F Royle.
\newblock {\em Algebraic graph theory}, volume 207.
\newblock Springer Science \& Business Media, 2013.

\bibitem{hahntardif}
Ge{\v{n}}a Hahn and Claude Tardif.
\newblock Graph homomorphisms: structure and symmetry.
\newblock In {\em Graph symmetry}, pages 107--166. Springer, 1997.

\bibitem{joseph1995isomorphism}
Anne Joseph.
\newblock The isomorphism problem for cayley digraphs on groups of prime-squared order.
\newblock {\em Discrete mathematics}, 141(1):173--183, 1995.

\bibitem{mansilla2004infinite}
S{\`o}nia~P Mansilla.
\newblock Infinite families of non-cayley vertex-transitive tournaments.
\newblock {\em Discrete mathematics}, 288(1):99--111, 2004.

\bibitem{marusic}
Dragan Maru{\v{s}}i{\v{c}}.
\newblock Vertex transitive graphs and digraphs of order $p^k$.
\newblock {\em North-Holland Mathematics Studies}, 115:115--128, 1985.

\bibitem{morris1999isomorphic}
Joy Morris.
\newblock Isomorphic cayley graphs on nonisomorphic groups.
\newblock {\em Journal of Graph Theory}, 31(4):345--362, 1999.

\bibitem{muzychuk2}
M~Muzychuk.
\newblock An elementary abelian group of large rank is not a ci-group.
\newblock {\em Discrete Mathematics}, 264(1):167--185, 2003.

\bibitem{sabidussisub}
Gert Sabidussi.
\newblock On a class of fixed-point-free graphs.
\newblock {\em Proceedings of the American Mathematical Society}, 9(5):800--804, 1958.

\bibitem{sabidussicoset}
Gert Sabidussi.
\newblock Vertex-transitive graphs.
\newblock {\em Monatshefte f{\"u}r Mathematik}, 68(5):426--438, 1964.

\bibitem{turner1967point}
James Turner.
\newblock Point-symmetric graphs with a prime number of points.
\newblock {\em Journal of Combinatorial Theory}, 3(2):136--145, 1967.

\bibitem{turner}
James Turner.
\newblock Point-symmetric graphs with a prime number of points.
\newblock {\em Journal of Combinatorial Theory}, 3(2):136--145, 1967.

\end{thebibliography}
\bibliographystyle{plain}

\end{document}